\documentclass[11pt]{article}
\usepackage[margin=1in]{geometry}
\usepackage{mathtools, nccmath,mathrsfs, esint}
\RequirePackage{amsthm,amsmath,amsfonts,amssymb, bbm,bm}
\RequirePackage[numbers]{natbib}
\RequirePackage[colorlinks,citecolor=blue,urlcolor=blue]{hyperref}
\RequirePackage{graphicx}
\usepackage[normalem]{ulem}
\numberwithin{equation}{section}
\theoremstyle{plain}
\newtheorem{theorem}{Theorem}[section]
\newtheorem{lemma}[theorem]{Lemma}
\newtheorem{prop}[theorem]{Proposition}
\newtheorem{corollary}[theorem]{Corollary}
\theoremstyle{remark}
\newtheorem{remark}[theorem]{Remark}
\theoremstyle{definition}
\newtheorem{defn}[theorem]{Definition}
\newtheorem{assumption}[theorem]{Assumption}

\newtheorem{example}[theorem]{Example}

\newcommand{\beq}{\begin{equation}}
\newcommand{\eeq}{\end{equation}}
\def\beqs#1\eeqs{%
    \begin{equation}\begin{split}%
    #1%
    \end{split}\end{equation}%
}
\def\beqsn#1\eeqsn{%
    \begin{equation*}\begin{split}%
    #1%
    \end{split}\end{equation*}%
}
\def\beqsj #1\eeqsj
{
    \begin{equation}
    \setlength{\jot}{10pt}
    \begin{split}
    #1
    \end{split}\end{equation}
}

\newcommand{\A}{\mathcal A}

\newcommand{\balph}{{\bm\alpha}}
\newcommand{\bbet}{{\bm\beta}}

\newcommand{\e}{\,\exp}
\newcommand{\E}{\mathbb E\,}

\newcommand{\even}{\mathrm{even}}
\newcommand{\iid}{\stackrel{\mathrm{i.i.d.}}{\sim}}
\newcommand{\ind}{\mathbbm{1}}

\renewcommand{\l}{\left}
\newcommand{\la}{\langle}
\newcommand{\lla}{\l\la}
\newcommand{\les}{\lesssim}

\renewcommand{\r}{\right}
\newcommand{\R}{\mathbb R}
\newcommand{\ra}{\rangle}
\newcommand{\Rad}{R}

\newcommand{\Rem}{\mathrm{Rem}}

\newcommand{\rr}{r}
\newcommand{\rra}{\r\ra}

\newcommand{\T}{{\intercal}}
\newcommand{\maus}[1]{\nu_{#1}}

\newcommand{\U}{{\mathcal U}}
\renewcommand{\u}{z}

\newcommand{\V}{{\mathcal V}}

\newcommand{\xmin}{x_0}

\title{The Laplace asymptotic expansion in high dimensions}
\author{
 Anya Katsevich\thanks{This work was supported by NSF grant DMS-2202963.}\\
  \texttt{akatsevi@mit.edu}
 }
\begin{document}
\maketitle

\begin{abstract}
We prove that the classical Laplace asymptotic expansion (AE) of $\int_{\R^d} g(x)e^{-n\u(x)}dx$, $n\gg1$ in powers of $n^{-1}$ extends to the high-dimensional regime in which $d$ may grow large with $n$. More specifically, we use new techniques suitable to the high-dimensional regime to derive an AE which formally coincides with the classical one because the terms are the same, but which now has a new small parameter. Namely suppose $\u$ and $g$ satisfy standard assumptions made in the classical expansion and additionally, that $\|\nabla^k\u\|\leq d^{\lceil k/2\rceil-2}$ for all $k\geq3$ and $\|\nabla^kg\|\leq d^{\lceil k/2\rceil}$ for all $k\geq0$ in a neighborhood of the unique minimizer of $\u$. Under these conditions, we show the new small parameter is $d^2/n$, in the sense that $|\Rem_L|\leq C_L(d^2/n)^L$ for each $L=1,2,3,\dots$, where $\Rem_L$ is the error upon truncating the expansion to its first $L$ terms. As an example, we show that these derivative bounds are satisfied with high probability for a random function $\u$ arising in a standard statistical model. In fact, the derivative bounds can be relaxed. We introduce an arbitrary large parameter $\tau$ into the bounds on the derivative norms to relax the requirement on their growth and still obtain a valid AE, in powers of a ``larger" small parameter $\tau^2d^2/n$. The AE is valid whenever $n\gg \tau^2d^2$.

To prove these results, we derive a new and very general nonasymptotic bound on $\Rem_L$ which is explicit in its dependence on $g,\u,d,n$. The bound requires smoothness of $g$ and $\u$ but holds with nearly no apriori restrictions on the magnitude of the derivative tensor norms. We show the bound is tight for each $L$ by proving a matching lower bound in the case that $\u$ is a simple quartic function and $g\equiv1$. When $d,\u,g$ are fixed and $n\to\infty$, our bound shows that $\Rem_L=\mathcal O(n^{-L})$. Thus our work \emph{subsumes the classical theory of the Laplace expansion}, and significantly extends it into the high-dimensional regime. This broadened applicability of the Laplace expansion is extremely useful for the many modern applications requiring the computation of high-dimensional Laplace integrals. In settings where the expansion is already in use, our precise and explicit error bound is valuable both for numerical estimates and theoretical analysis, especially near the boundary of applicability of the expansion.

\end{abstract}

\section{Introduction}
The asymptotic expansion of Laplace-type integrals $\int_{\R^d}g(x)e^{-n\u(x)}dx$ as $n\to\infty$ is one of the cornerstones of asymptotic analysis. This result states~\cite[Chapter 9]{wongbook} that under certain smoothness and growth conditions on $\u$ and $g$, and assuming $\u$ has a unique global minimizer $\xmin$, there are coefficients $A_{2k}$ such that
\begin{align}\label{d-dim-laplace}
\int_{\R^d} g(x)e^{-n\u(x)}dx &=  \frac{(2\pi/n)^{d/2}}{e^{n\u(\xmin)}\sqrt{\det H}}\l(g(\xmin )+\sum_{k=1}^{L-1} A_{2k}n^{-k} + \Rem_L\r),\\
\Rem_L&= \mathcal O(n^{-L}),\quad n\to\infty.\label{RemL}
\end{align}Here, $H=\nabla^2\u(\xmin)$. 
Explicit expressions for the coefficients $A_{2k}$ are given in~\cite{nemes2013explicit,wojdylo2006computing} in the one-dimensional case, and have (surprisingly) only recently been obtained in the general $d$-dimensional case, by~\cite{kirwin2010higher}.

Laplace-type integrals arise in a number of contexts both within mathematics and in scientific and statistical applications, and the expansion~\eqref{d-dim-laplace}-\eqref{RemL} is an indispensable tool to both numerically compute and mathematically analyze these integrals. For example, Laplace-type integrals arise in statistical mechanics~\cite{ellis1982laplace,bach2000correlation}, optimization~\cite{carrillo2021consensus,tibshirani2024laplace}, the analysis of random chaos~\cite{korshunov2015asymptotic}, the computation of rare event probabilities e.g. in stochastic systems~\cite{schorlepp2023scalable} and in reliability engineering~\cite{rackwitz2001reliability}, and in statistics~\cite{rue2017bayesian}. See Section~\ref{sec:app} for more references in statistics.

The theory of the Laplace expansion (LE) dates back to Laplace in 1774 and is essentially fully understood in the classical regime in which dimension $d$ is fixed and $n\to\infty$. In fact, a full theory of the LE also exists in the \emph{infinite}-dimensional regime, due to~\cite{arous1988methods}. In this context, the domain of integration is path space, e.g. the space of solutions to a stochastic differential equation (SDE), and the small parameter $n^{-1}$ is the noise level of the SDE.

However, there is an intermediate and largely unexplored regime between the cases $d=\mathcal O(1)$ and $d=\infty$. Is there an expansion when $d$ grows large with $n$? Beyond the mathematical significance, resolving this question is of great importance in many modern applications, particularly within statistics, where computing a high-dimensional Laplace integral is required. 

Deriving and even properly defining a suitable notion of high-dimensional expansion goes hand in hand with a second goal: explicitly controlling the LE remainder. Indeed, since $\u$ and $g$ themselves depend on the growing $d$, these functions can contribute non-negligible factors to the remainder. Thus one must fully quantify how the remainder depends on all problem parameters $\u,g,d,n$. Obtaining such an explicit bound, which is at the same time tight, is challenging when $d\gg1$, and has never been done before. Here, we do precisely this: we give the first ever theory of the LE in the high-dimensional regime, including explicit and tight control of the remainder. The non-asymptotic bounds we derive allow us to characterize a broad, flexible class of $d,n,\u,g$ admitting a valid LE.

\subsection{Main results}\label{intro:results}
The existing, classical expansion~\eqref{d-dim-laplace}-\eqref{RemL} cannot formally be used as a starting point to derive a high-dimensional expansion. Indeed, there is no guarantee that the remainder $\Rem_L$ remains small when $d$ grows large with $n$, and therefore also no apriori guarantee that the terms of a high $d$ expansion should coincide with the terms $A_{2k}n^{-k}$ in~\eqref{d-dim-laplace}. 

Nevertheless, by ``starting from scratch" and using new techniques suitable to the high-dimensional regime, we arrive at a decomposition of the Laplace integral which \emph{does} formally coincide with~\eqref{d-dim-laplace}. But in place of the remainder bound~\eqref{RemL}, we obtain a new nearly explicit bound on $\Rem_L$ which takes the following form:
\beq\label{RemL-main-bound}
|\Rem_L|\les_L  \A_{2L}(\u,g)(d^2/n)^L.\eeq This holds under certain assumptions on the smoothness and growth at infinity of $\u$ and $g$; see Theorem~\ref{thm:expand} for the precise statement. In this bound, $\les_L$ indicates a suppressed non-explicit constant $C_L$ depending only on $L$, and $\A_{2L}(\u,g)$ is a fully explicit function of the derivatives of $\u$ and $g$ in a neighborhood of $\xmin$, the point of minimum. In the classical setting in which $\u$, $g$, and $d$ are fixed and $n\to\infty$, we recover from~\eqref{RemL-main-bound} the fact that $\Rem_L=\mathcal O(n^{-L})$. Thus \emph{our theory subsumes the classical theory of the LE}. 

More importantly, we significantly extend it. Informally, we have shown that the expansion~\eqref{d-dim-laplace} to order $L$ is useful for any $\u,g,d,n$ for which $\A_{2L}(\u,g)(d^2/n)^L$ is sufficiently small. This simple fact is powerful in practice, where decisions must be made non-asymptotically, based on a single given $\u,g,d,n$. The common knowledge that the LE is safe to apply as long as $n$ is large enough and $d$ is not too large is of course true. Yet frequently in modern applications, the setting is not sufficiently clear-cut that one can come to a decisive conclusion based on the size of $n$ and $d$ alone. Now, we have a complete characterization of when the LE is safe to apply, based on the magnitude of $\A_{2L}(\u,g)(d^2/n)^L$. 
We discuss the practical ramifications of this further in Section~\ref{sec:app}.
 
From a more formal mathematical perspective,~\eqref{RemL-main-bound} can be used to show that~\eqref{d-dim-laplace} constitutes a true, generalized asymptotic expansion~\cite[Chapter 1]{wongbook}, in the following sense. Consider any sequence of dimensions $d_n$ for which $d_n^2/n\to0$ as $n\to\infty$, and functions $\u_n,g_n:\R^{d_n}\to\R$, for which $\A_{2L}(\u_n,g_n)\leq C_L$ as $n\to\infty$, for each $L=1,2,3,\dots$. Then~\eqref{d-dim-laplace} is a generalized asymptotic expansion in powers of the small parameter $d_n^2/n$, meaning precisely that $|\Rem_L|\les_L (d_n^2/n)^L$ for each $L$.
More generally, if $\A_{2L}(\u_n,g_n)\les_L\tau_n^{2L}$ for some potentially large $\tau_n$, then~\eqref{RemL-main-bound} implies that we have an expansion in powers of $\tau_n^2 d_n^2/n$, provided this quantity goes to zero with $n$. For example if $\tau_n=\sqrt {d_n}$ and $d_n^3/n\to0$, we get an expansion in powers of the small parameter $d_n^3/n$. We could even have that $\tau_n\ll 1$ in special cases. From now on we typically omit the $n$ subscript to lighten the notation.

Importantly, we also formulate simple conditions on $\u,g$ which imply the necessary bounds on $\A_{2L}(\u,g)$. Here and below, let $\|\nabla^kf\|_H$ be shorthand for $\sup_{x\in\U}\|\nabla^kf(x)\|_H$, where $\U$ is a small neighborhood of $\xmin$ and $\|\cdot\|_{H}$ is an $H$-weighted operator norm, where $H=\nabla^2\u(\xmin)$. Then, in order to have $\A_{2L}(\u,g)\les_L1$, it suffices that
\beqs\label{simple:tau1}
\|\nabla^k\u\|_H&\les_k d^{\left\lceil\frac k2\right\rceil -2}\qquad \forall\;k=3,4,\dots,2L+2,\\
\|\nabla^kg\|_H&\les_k d^{\left\lceil\frac k2\right\rceil}\qquad \forall\;k=0,1,\dots,2L.
\eeqs
See Lemma~\ref{lma:powers} for analogous conditions which imply $\A_{2L}(\u,g)\les\tau^{2L}$ in the case of general $\tau$. 

To understand why $\A_{2L}(\u,g)$ can be bounded by a constant even as the derivative tensor norms of $g$ and $\u$ grow with $d$, we note that $\A_{2L}(\u,g)$ takes the following form, upon some simplification:
\beq\label{slack}\A_{2L}(\u,g)=a_{L,0}(\u,g)+a_{L,1}(\u,g)d^{-1} +\dots+ a_{L,2L}(\u,g)d^{-2L}.\eeq The quantities $a_{L,j}(\u,g)$ have no explicit $d$ dependence, and are given by polynomials of $\|\nabla^k\u\|_H$, $\|\nabla^kg\|_H$. For example, $a_{1,1}(\u,g)=\|\nabla g\|_H\|\nabla^3\u\|_H + \|\nabla^2g\|_H$. The structure~\eqref{slack} shows that $\A_{2L}(\u,g)$ has built-in slack, meaning  $a_{L,j}(\u,g)$ can be as large as $d^j$, $j=0,\dots, 2L$ in order for $\A_{2L}(\u,g)$ to be bounded by a constant. Translating the conditions $a_{L,j}(\u,g)\les d^j$, $j=0,\dots,2L$ into conditions on $\|\nabla^k\u\|_H$, $\|\nabla^kg\|_H$ gives~\eqref{simple:tau1}. In Section~\ref{sec:log} we give an example where this slack is used to its full capacity.



\paragraph{Tightness of the bound~\eqref{RemL-main-bound}.}We show that our bound on $\Rem_L$ is tight in the following sense. Fix any sequence $d_n$ such that $d_n^2/n\to0$. We exhibit a sequence of functions $\u_n,g_n:\R^{d_n}\to\R$ such that there are constants $c_L, C_L>0$ for which $|\Rem_L|\geq c_L(d_n^2/n)^L$ and $\A_{2L}(\u_n,g_n)\leq C_L$ as $n\to\infty$, for all $L=1,2,3,\dots$. Thus~\eqref{RemL-main-bound} gives the upper bound $|\Rem_L|\leq C_L'(d_n^2/n)^L$, and this upper bound matches our lower bound up to a constant depending only on $L$. The functions are very simple: $\u_n(x)=\|x\|^4/24+\|x\|^2/2$ and $g_n(x)=1$. 
 
\paragraph{Application to a random $\u$ arising in statistics.} In Section~\ref{sec:log}, we apply our results to analyze the expansion for a particular choice of $\u$ arising as the log likelihood in generalized linear models~\cite{agresti2015foundations}, a common type of statistical model. The function is random due to its dependence on randomly drawn data. We do not specify $g$ but leave it to be any well-behaved function, which unlike $\u$ is usually not random. See Section~\ref{sec:log} for more details. 

Beyond its practical relevance, the example is compelling because we prove that the function $\u$ uses up all the available slack. Namely, the upper bounds we prove on the derivative tensor norms match the maximum growth with $d$ allowed by~\eqref{simple:tau1}. Our bounds, which hold with high probability, imply that the $\A_{2L}(\u,g)$ are bounded by deterministic constants $C_L$ on this high probability event, for suitable functions $g$. On this same event,~\eqref{d-dim-laplace} is then a generalized asymptotic expansion in powers of $d^2/n$. See Section~\ref{sec:log} for precise statements.

\paragraph{Proof approach.} The main innovation in our proof approach is to express the remainder in terms of Gaussian expectations of certain tensor inner products, and then to leverage the powerful theory of high-dimensional Gaussian concentration to prove tight bounds on these expectations. 

As noted above, our expansion subsumes the classical theory, so in particular, our proof also applies to the case of fixed $d$. To our knowledge, our proof is new even in this fixed $d$ regime. A natural question is whether developing a new proof of this classical result is really necessary, since it may be possible to handle the case $d\gg1$ simply by studying the error estimates obtained in existing proofs in the fixed $d$ case. This is unlikely. Indeed, in classical derivations, the estimates are typically inexplicit or become unwieldy when $d\gg1$~\cite[Chapter 9]{wongbook}. 

Like us, recent works such as ~\cite{inglot2014simple,lapinski2019multivariate,kolokoltsov2020rates} (discussed below) which do obtain simple and explicit bounds on $\Rem_1$ develop new proof techniques rather than adapting classical derivations. However, these techniques lead to bounds which are too coarse in their dimension dependence. Thus our approach --- particularly the new representation of the remainder, enabling the use of Gaussian concentration ---  is instrumental to obtain \emph{tight} bounds in the high-dimensional context. 

Finally we mention that, while we have not quantified the suppressed constant $C_L$ in~\eqref{RemL-main-bound}, our proof technique is based on explicit calculations in which one can certainly keep track of the constants. We leave this to future work.

See Section~\ref{sec:outline} for more details on our proof approach.


\subsection{State of the art in high dimensions and theoretical significance}\label{intro:stateofart}
We now review recent literature on the LE in high dimensions and the related problem of obtaining explicit remainder bounds. We then explain the theoretical significance of our own work against this backdrop of the state of the art. 


The works~\cite{inglot2014simple,lapinski2019multivariate,kolokoltsov2020rates} have focused on obtaining simple and explicit error bounds on $\Rem_1$. Although the authors are not directly interested in the high-dimensional regime, the fact that the bounds are explicit allows one to apply them in the case $d\gg1$ as well. In~\cite{lapinski2019multivariate,kolokoltsov2020rates}, the bounds on $\Rem_1$ blow up exponentially in $d$, due to the presence of expressions of the form $\det(\nabla^2\u(\xmin))/\inf_{x\in\U}\lambda_{\min}(\nabla^2\u(x))^d$, using our notation. Here, $\U$ is a neighborhood of $\xmin$. The work~\cite{inglot2014simple} considers the case $g\equiv1$. The author's bound on $\Rem_1$ is dominated by $K_1/n$, which scales as $d^6/n$ if $\max_{i,j,k}|\partial^3_{ijk}\u(\xmin)|$ is a constant. 

The works~\cite{barber2016laplace,tang2021laplace} are both explicitly interested in the high $d$ case, both studying $\Rem_1$ for $g\equiv1$. The work~\cite{barber2016laplace} considers a statistical context in which $\u$ is a particular random function similar to the one studied in Section~\ref{sec:log}. The authors show that in this case, $|\Rem_1| \les\sqrt{(d\log n)^3/n}$ with high probability.~\cite{tang2021laplace} proves bounds on $\Rem_1$ for a broader class of functions $\u$. It is not possible to directly compare our assumptions with those of~\cite{tang2021laplace}. However, under conditions which imply $|\Rem_1|\les d^2/n$ in our work, the bound of~\cite{tang2021laplace} gives $|\Rem_1|\les d^{3}/n$; see Example~\ref{ex:L1} for more details.



There have also been a few non-rigorous works on the full high-dimensional LE; see~\cite{shun1995laplace,smith2011spatial,ogden2021error}. Of these,~\cite{shun1995laplace} has been particularly influential. There, Shun and McCullagh study the terms of the expansion~\eqref{d-dim-laplace} for $g\equiv1$, in the context of two statistical models. In the first, the ``crossed random effects model" mentioned in Section~\ref{sec:app} below, it holds $d=\mathcal O(n^{1/2})$. The authors argue that $A_2n^{-1}$ is then $\mathcal O(1)$ as $n\to\infty$~\cite[Section 4.2]{shun1995laplace}. The second case is a generalized linear model, the same as the one we study in Section~\ref{sec:log} and discussed above. Shun and McCullagh claim that in this setting, the term $A_2n^{-1}$ is not small if $d^3/n$ is not small, and therefore the standard LE is not valid~\cite[Section 6]{shun1995laplace}.  The impression that $d^3\ll n$ is the boundary of applicability of the LE has persisted until nearly the present day, with works such as~\cite{han2024enhanced} continuing to cite~\cite{shun1995laplace} for guidance on when the LE can be applied. The aforementioned rigorous works~\cite{barber2016laplace,tang2021laplace} showing $\Rem_1$ is small when $d^3\ll n$ only seem to support this impression. Given the influence of~\cite{shun1995laplace} in the community, we have chosen to discuss the authors' claims in more detail in Appendix~\ref{app:shun}, explaining exactly why their arguments leading to $d^3\ll n$ are not sharp.

Beyond the specifics of how large $n$ should be relative to $d$, the work~\cite{shun1995laplace} also set the tone about the feasibility of a general theory of the LE in high dimensions. The authors claim that different functions $\u$ may necessitate different groupings of terms in the standard expansion, and comment that ``it does not seem feasible at present to develop useful general theorems for approximating arbitrary high-dimensional integrals". Since then, only the case $L=1$ has been understood in any capacity in the high $d$ regime.

Thus our work goes much farther than the prior state of the art. We have shown that a general theory \emph{is} possible by rigorously deriving a high-dimensional LE, which \emph{is} in fact valid under quite general conditions. We have obtained nearly explicit bounds on $\Rem_L$ for arbitrary order $L$, and tightened the dimension dependence of the bound on $\Rem_1$ compared to the previous works. Moreover, our bounds for arbitrary $L$ are optimal, in the sense described in Section~\ref{intro:results}. Our work also highlights the important role played by the derivatives of $\u$ and $g$, in addition to $d$ and $n$ themselves. Indeed, different growth rates of the derivatives of $\u$ and $g$ give rise to different small parameters $(\tau d/\sqrt n)^2$, as discussed in Section~\ref{intro:results}. Interestingly and conveniently enough, these different asymptotic behaviors of the remainders all coexist: they describe the error after truncation of the same, original LE~\eqref{d-dim-laplace}, with \emph{no changes to the formulas for the terms $A_{2k}n^{-k}$}. This is convenient in practice because it means that no new formulas need to be implemented. 

This is also convenient in theory, since it means that separate analyses are not needed in different regimes. To determine the size of our bounds on $\Rem_L$ for a particular family of $\u,g$, one need only compute the operator norms of the derivative tensors $\nabla^k\u$, $\nabla^kg$. 
This way, our results have provided a clear template for how the asymptotic behavior of the LE in different settings can be determined.

To conclude, let us mention a related line of work. In Bayesian statistics, one is interested in approximating ratios of two Laplace integrals: $\int g(x)e^{-n\u(x)}dx/\int e^{-n\u(x)}dx$, where $g$ is not necessarily smooth near $\xmin$, so we cannot use a Laplace expansion of the numerator. This ratio can also be written as $\int gd\pi$, where $\pi$ is a probability distribution with density proportional to $e^{-n\u(x)}$. The approach taken to this problem is to approximate $\pi$ by a Gaussian $\hat\pi$ and to take $\int gd\pi\approx\int gd\hat\pi$. The Gaussian $\hat\pi$ stems from replacing $\u$ by its quadratic Taylor expansion about $\xmin$, which is exactly the principle behind the order $L=1$ Laplace expansion. There have been many works on bounding the accuracy of the approximation $\pi\approx\hat\pi$ in high dimensions; see e.g.~\cite{spok23,katskew, bp, fischer2022normal, dehaene2019deterministic, huggins2018practical, helin2022non}.

\subsection{Practical implications}\label{sec:app} 
Laplace-type integrals are ubiquitous in applications, particularly within statistics. One of the primary ways such integrals arise in statistics is through the following optimization problem:
\beq\label{MMLE}\max_{\theta}\int_{\R^{d}}g(x)e^{-n\u_\theta(x)}dx,\eeq where the maximization can be over either a discrete or continuous set of $\theta$'s, and $d$ may also vary with $\theta$. The interpretation of the integral varies depending on the statistical \emph{framework} used (Bayesian or frequentist). In the Bayesian framework, the problem~\eqref{MMLE} is solved for the purpose of Bayesian model selection, and the integral is known as the model evidence~\cite{wasserman2000bayesian},~\cite[Chapter 5]{ando2010bayesian}. In the frequentist framework,~\eqref{MMLE} constitutes a marginal maximum likelihood estimation problem, and the integral is a marginal likelihood~\cite{ruli2016improved}. The structure of the problem~\eqref{MMLE} also depends on the particular statistical \emph{model} being studied; this affects the domain of $\theta$, the form of $\u_\theta$ and $g$, the scaling of $n$ with $d$, and the possible dependence of $d$ on $\theta$. In some models, the objective is actually given by a sum of logarithms of several Laplace-type integrals, each of which depends on $\theta$~\cite[Equation (2)]{hall2020fast}. We consider all of these cases to be problems of the general type~\eqref{MMLE}.

These problems are particularly computationally intensive when the integration is high-dimensional, which occurs often in modern applications. For example,~\cite{immer2021scalable} studies~\eqref{MMLE} in the context of deep learning, where the integration is over the space of hundreds of parameters of a neural network. Thus efficient numerical techniques are essential, and as we explain shortly, the LE is a particularly well-adapted tool to solve~\eqref{MMLE}. 

The LE was introduced to the statistical community in the landmark paper~\cite{tierney1986accurate}, and it has been a popular approach to simplify the calculation of Laplace-type integrals in statistics ever since. The LE to finite order is typically called the ``Laplace approximation" in statistics, or LA for short, so we adopt this terminology here as well. For the purpose of accurately computing a \emph{single} Laplace-type integral, it may be advantageous to use the alternative, arguably more standard method of Markov Chain Monte Carlo (MCMC)~\cite{robert1999monte}, rather than the LA. But in the context of the optimization problem~\eqref{MMLE}, the LA enjoys significant advantages over MCMC.  The issue with MCMC is that its output is numerical. Thus MCMC only gives black box access to the integral for each $\theta$, whereas the LA outputs an analytic expression for the approximate integral as a function of $\theta$. This is extremely useful when solving the optimization problem in $\theta$.


But application of the LA to solve~\eqref{MMLE} has been hampered by the lack of a high-dimensional theory. To illustrate this, consider the Bayesian Information Criterion (BIC), an extremely popular and very simple method~\cite{BIC} based on the LA to solve a version of~\eqref{MMLE}. Writing about the limitations of the BIC, Drton and Plummer state, ``for models with a large number of predictors [large $d$] it is no longer clear that some version of the BIC, or perhaps rather a Laplace approximation, accurately approximates a marginal likelihood"~\cite[p. 377]{drton2017model}. In fact, the authors go on to cite the influential work~\cite{shun1995laplace} discussed above, saying that ``Shun and McCullagh (1995) have shown that the model dimension cannot be too large for a Laplace approximation to be useful". 

There has been limited progress in applying the LA in high dimensional statistical settings. For example, the aforementioned work~\cite{barber2016laplace} gives conditions under which the LA can be used to accurately solve~\eqref{MMLE}, in the same context as the one in which the BIC is applied. In particular, a key condition is that $d\ll n^{1/3}$. But the present work shows that for suitable $\u$ and $g$, the LA can be used for $d$ as large as $d\ll \sqrt n$. \emph{Thus the first major practical implication of our results is to extend the use of the LA farther out into the high-dimensional regime, where it is sorely needed. }

Our results are also useful in more moderate dimensions in which the LA is already in active use. Indeed, any applications using the LA in which precise error estimates are required can benefit from the bound~\eqref{RemL-main-bound}, particularly once the omitted constants are obtained. Furthermore, our bounds are especially useful close to the boundary of applicability of the LA, where the accuracy of the method is in question. We now describe one such compelling example. 

The example pertains to solving~\eqref{MMLE} in a class of statistical models called mixed and random effects models~\cite{searle2009variance}. A large number of works have proposed to solve~\eqref{MMLE} for these models by using the LA expanded out to different orders, or using variants of the LA. The foundational papers on this topic include~\cite{solomon1992nonlinear,breslow1993approximate,liu1993heterogeneity, breslow1995bias}. The idea has been further developed or studied e.g. in the works~\cite{raudenbush2000maximum, joe2008accuracy, rizopoulos2009fully, smith2011spatial, bianconcini2012estimation, kim2013logistic, li2020multilevel}. The performance of the LA has been studied primarily with numerical simulation or nonrigorous arguments, and reports on its accuracy are varied. For example, the performance is shown to be good in~\cite{joe2008accuracy} but poor relative to other methods in~\cite{hall2020fast}. 

This ambiguity is unsurprising, because there are a variety of statistical settings within mixed and random effects models (i.e. particular regimes of $\u_\theta,g,d,n$) in which the LA is applied to~\eqref{MMLE}. Most likely, the LA error is simply too large to be able to accurately solve the optimization problem in some of these settings. This can be established in each case with the help of our precise LA error bound $\A_{2L}(\u,g)(d^2/n)^L$. For instance, our theory already gives strong indication that the LA is likely too coarse in the setting of crossed random effects models, where typically $d=\mathcal O(\sqrt n)$~\cite{shun1995laplace,han2024enhanced}. Determining this definitively would also require studying $\A_{2L}(\u,g)$, by bounding the derivative tensor norms of $\u$ and $g$.

It is clear that in all but the simplest, low $d$ cases, \emph{the full force of our non-asymptotic error bounds is needed to rigorously quantify the performance of the LA to solve~\eqref{MMLE}}.

\subsection*{Organization.}In Section~\ref{sec:main} we state our assumptions and main results on the high-dimensional Laplace asymptotic expansion. In Section~\ref{sec:log}, we apply the theory to a function $\u$ stemming from a standard statistical model. Section~\ref{sec:proof} outlines the proof. The Appendix contains supplementary lemmas and the comparison with~\cite{shun1995laplace}.

\section{Main results}\label{sec:main}In this section we present our main results. We start in Section~\ref{subsec:assume} by stating our assumptions and introducing important quantities. In Section~\ref{subsec:expand} we present our high-dimensional expansion of $\int ge^{-n\u}$ in its most general form. Then in Section~\ref{sec:ae} we present assumptions on the growth of the derivatives of $g$ and $\u$ under which our high-dimensional expansion reduces to a true asymptotic expansion in powers of a small parameter. We also present an example with a quartic function $\u$ for which our remainder bounds are tight. 

\subsection*{Notation}\label{notation} 
We start with notation pertaining to tensors, and stress the importance of the below distinction between norms of vectors in $\R^d$ and norms of linear forms on $\R^d$. For a vector $x\in\R^d$, and a positive definite matrix $H\succ0$, we let
\beq\label{xH}
\|x\|_H = \sqrt{x^\top H x}.
\eeq
A tensor $T$ of order $k$ is an array $T=(T_{i_1i_2\dots i_k})_{i_1,\dots,i_k=1}^d$, and we also consider tensors to be $k$-variate linear forms on $\R^d$, with
$$T(u_1,\dots, u_k) = \la T, u_1\otimes\dots\otimes u_k\ra = \sum_{i_1,\dots,i_k=1}^dT_{i_1i_2\dots i_k}(u_1)_{i_1}(u_2)_{i_2}\dots (u_k)_{i_k}.$$ 
 For two order $k$ tensors $T$ and $S$ we let $\la T, S\ra$ be their entrywise inner product.  
The operator norm of $T$ with respect to the $\|\cdot\|_H$ norm on $\R^d$ is given by
\beq\label{TH}\|T\|_H=\sup_{\|u_1\|_H=\dots=\|u_k\|_H=1}\la T, u_1\otimes\dots\otimes u_k\ra.\eeq
We say $T$ is symmetric if $T_{i_1\dots i_k}= T_{j_1\dots j_k}$, for all permutations $j_1\dots j_k$ of $i_1\dots i_k$. By~\cite[Theorem 2.1]{symmtens}, for symmetric tensors we have $\|T\|_H=\sup_{\|u\|_H=1}\la T, u^{\otimes k}\ra.$ In particular, 
\beqs\label{T-form}
\|T\|_H &= \|H^{-1/2}T\|,\qquad k=1\\
\|T\|_H &= \|H^{-1/2}TH^{-1/2}\|,\qquad k=2,
\eeqs where in the first line $\|\cdot\|$ is the regular Euclidean norm on $\R^d$ and in the second, $\|\cdot\|$ is the regular matrix operator norm and $T$ is a symmetric $d\times d$ matrix. 

Finally, we note an important distinction between~\eqref{xH} and~\eqref{T-form} in the case that $k=1$. In~\eqref{T-form}, the ``tensor" or linear form $T$ can of course be represented as a vector $T\in\R^d$, but the meaning of $\|T\|_H$ in~\eqref{T-form} is different from the meaning in~\eqref{xH}. The only time we treat vectors as linear forms is when the vector is a gradient of a function. Thus $\|\nabla f\|_H$ always means $\|H^{-1/2}\nabla f\|$ rather than $\|H^{1/2}\nabla f\|$.

For $a,b>0$, we say that $a\les_L b$ if $a\leq C_Lb$ for some constant $C_L>0$ depending only on $L$. For $a\in\R$ and $b>0$, we say that $a\asymp_L b$ if $c_Lb \leq |a|\leq C_Lb$ for some constants $c_L, C_L>0$ depending only on $L$. The measure $\gamma$ denotes the standard Gaussian measure on $\R^d$: $\gamma(dx)=(2\pi)^{-d/2}e^{-\|x\|^2/2}dx$. The function $\ind_\U:\R^d\to\R$ is the indicator of the set $\U$, and 
$$\epsilon = d/\sqrt n.$$


\subsection{Assumptions and important quantities}\label{subsec:assume}
\begin{assumption}[Minimizer of $\u$]\label{assume:min}
The function $\u$ has a unique global minimizer at a point $\xmin $. Furthermore, $\u$ is twice differentiable in a neighborhood of $\xmin $ to be specified, and $H:=\nabla^2\u(\xmin )\succ0$.
\end{assumption} For a fixed radius $\Rad\geq1$ to be specified later, we let $\U_H(\Rad)=\{x\; :\|x-\xmin \|_H < \Rad\sqrt d\}$ and $n^{-1/2}\U_H(\Rad)=\{x/\sqrt n\; : \; x\in\U_H(\Rad)\}$. Roughly speaking, when $\log n\ll d$ we can take $\Rad$ to be an absolute constant, and when $d$ is small, $\Rad\sim\log n$. See Section~\ref{sec:ae} for further considerations about the choice of $\Rad$.

\begin{assumption}[Regularity of $g$ and $\u$]\label{assume:CL}
It holds 
$$g\in C^{2L}\left(n^{-1/2}\bar\U_H(\Rad)\right),\qquad \u\in C^{2L+2}\left(n^{-1/2}\bar\U_H(\Rad)\right).$$ \end{assumption} Here, $\bar\U_H(\Rad)$ denotes the closure of $\U_H(\Rad)$, i.e. the set $\{x\; :\|x-\xmin \|_H\leq\Rad\sqrt d\}$. This regularity condition means that the derivatives $\nabla^kg$, $k=0,\dots,2L$ and $\nabla^k\u$, $k=0,\dots,2L+2$ exist and are uniformly continuous in the open set $n^{-1/2}\U_H(\Rad)$, and can therefore be uniquely extended to continuous functions up to the boundary.

\begin{assumption}[Growth of $g$ and $\u$ at infinity]\label{assume:tail}
For all $\|x-\xmin\|_H\geq\Rad\sqrt{d/n}$, it holds 
\begin{align}
\u(x) - \u(\xmin ) &\geq \frac13\sqrt{d/n}\|x-\xmin \|_H ,\label{u-imply}\\
|g(x)| &\leq \exp\left(\frac{\sqrt{dn}}{12}\|x-\xmin \|_H\right),\label{g-imply}
\end{align}
\end{assumption}
\begin{remark}
The constants $1/3$ and $1/12$ are for convenience, but these specific values do not create any essential limitation on $g$ and $\u$. This is because as $n$ increases, both bounds become more lax. Also, suppose $g$ does not satisfy~\eqref{g-imply}, but there is some $C>1$ such that $|g(x)| \leq C\exp((\sqrt{dn}/12)\|x-\xmin \|_H)$ for all $\|x-\xmin\|_H\geq\Rad\sqrt{d/n}$. Then our main Theorem~\ref{thm:expand} below still goes through, with all error bounds multiplied by $C$. This is because $g$ enters the Laplace integral linearly. 

Note also that the condition~\eqref{u-imply} is trivially satisfied for $x$ close to $\xmin$. Indeed, a Taylor expansion of $\u$ around $\xmin $ gives that $\u(x)-\u(\xmin )\approx \|x-\xmin \|_H^2/2$ when $x$ is sufficiently close to $\xmin $. Therefore, there is some $C$ such that
\beqs\label{uminx}
\u(x)-\u(\xmin )\geq \frac13\|x-\xmin \|_H^2\geq \frac{R}{3}\sqrt{d/n}\|x-\xmin \|_H\geq\frac13\sqrt{d/n}\|x-\xmin\|_H\\
\eeqs
 for all $\Rad\sqrt{d/n}\leq \|x-\xmin \|_H\leq C$. As a result,~\eqref{u-imply} should be interpreted as a condition on the growth of $\u$ when $\|x-\xmin \|_H$ is larger than some constant. If $\u$ is convex then the situation is even simpler, since the infimum of $(\u(x)-\u(\xmin ))/\|x-\xmin \|_H$ over all $\|x-\xmin\|_H\geq\Rad\sqrt{d/n}$ is achieved on the boundary, i.e. for some $x$ such that $\|x-\xmin\|_H=\Rad\sqrt{d/n}$. Thus the fact that~\eqref{uminx} is satisfied for $\|x-\xmin\|_H=\Rad\sqrt{d/n}$ suffices to show that~\eqref{u-imply} holds for all $\|x-\xmin\|_H\geq\Rad\sqrt{d/n}$. 
 
 See Lemma~\ref{lma:cvx} for a more precise statement of the above argument.
\end{remark}

\begin{remark}[Relaxing Assumption~\ref{assume:tail}] We never make separate use of~\eqref{u-imply} and~\eqref{g-imply}, but only together, in the form
\beq\label{g-u-combo}|g(x)|\exp\l(-n[\u(x)-\u(\xmin )]\r)\leq \exp\left(-\frac14\sqrt{dn}\|x-\xmin \|_H\right)\eeq for all $\|x-\xmin \|_H\geq\Rad\sqrt{d/n}$. Thus it suffices to assume~\eqref{g-u-combo}. In fact, even~\eqref{g-u-combo} is not essential, because we only use it to show (see Lemma~\ref{aux:gamma}) that
\beqs
|\maus{\U^c}|:=(2\pi)^{-d/2}&\left|\int_{\|x\|\geq\Rad\sqrt d}f(x/\sqrt n)e^{-nv(x/\sqrt n)}dx\right|\\
&\leq (2\pi)^{-d/2}\int_{\|x\|\geq\Rad\sqrt d}\exp(-\sqrt d\|x\|/4)dx \leq de^{-\Rad d/16}.
\eeqs Here, $v(x)=\u(\xmin+H^{-1/2}x)-\u(\xmin)$, and $f(x)=g(\xmin +H^{-1/2}x)$. The term $\maus{\U^c}$ arises in our error bounds, and should be exponentially small. Any other condition on the growth of $\u,g$ at infinity which leads to $|\maus{\U^c}|\leq de^{-\Rad d/16}$ (or some similar exponentially small bound) would suffice. 
\end{remark}


\begin{defn}[Derivative bounds]\label{def:cvf}
When a continuous derivative of order $k$ exists, and for $r=0,r=\Rad$, let
\beqsn
c_{k}(r) &=\sup_{\|y\|_H\leq r\sqrt{d/n}}\;\|\nabla^k\u(\xmin+y)\|_H\qquad k=3,4,5,\dots,\\
c_{k,g}(r)& =\sup_{\|y\|_H\leq r\sqrt{d/n}}\; \|\nabla^kg(\xmin+y)\|_H\qquad k=0,1,2,\dots .
\eeqsn
When continuous derivatives of the necessary order exist, let
\beqs
\bar c_k(0)&=c_k(0),\qquad\bar c_{k}(\Rad)=c_{k}(0) + \epsilon c_{k+1}(\Rad),\qquad k=3,5,7,\dots,\\
\bar c_{k,g}(0)&=c_k(0),\qquad\bar c_{k,g}(\Rad)=c_{k,g}(0) + \epsilon c_{k+1,g}(\Rad),\qquad k=1,3,5,\dots .
\eeqs where $\epsilon=d/\sqrt n$.
\end{defn} 
\begin{defn}[Downweighted derivatives]\label{def:alph} For $r=0$ and $r=\Rad$, and provided continuous derivatives of necessary order exist, let
\begingroup
\addtolength{\jot}{1em}
\beqs
\alpha_{k}(r)&=d^{1-\left\lceil\frac k2\right\rceil}\times 
\begin{dcases}
\bar c_{k+2}(r), &k=1,3,5,\dots,\\
c_{k+2}(r), &k=2,4,6,\dots.
\end{dcases}\\
\alpha_{k,g}(r) &=d^{-\left\lceil\frac k2\right\rceil}\times \begin{dcases}
c_{k,g}(r), &k=0,2,4,\dots,\\
\bar c_{k,g}(r), &k=1,3,5,\dots.
\end{dcases}
\eeqs
\endgroup
\end{defn}
\begin{remark}[Role of $\alpha$'s, and ``getting rid of" $\bar c_k,\bar c_{k,g}$]\label{rk:bar}
Essentially, our bound on the $L$th remainder of the Laplace expansion is given by a polynomial of the $\alpha$'s multiplied by $(d^2/n)^L$. Thus if the $\alpha$'s are each bounded by a constant, then the $L$th remainder will be bounded by $(d^2/n)^L$. The $\alpha$'s being bounded by a constant translates to the $c$'s and $\bar c$'s being bounded by certain powers of $d$. It would be somewhat cumbersome to check conditions on $\bar c_k$ and $\bar c_{k,g}$, since they involve both the $k$th and $(k+1)$st derivative. But we avoid bounding $\bar c_k$ and $\bar c_{k,g}$ entirely. We show that it suffices to bound $c_k(0)$ for odd $k$ and $c_k(\Rad)$ for even $k$ to obtain the necessary bounds on $\bar c_k$ and $\bar c_{k,g}$. See Lemma~\ref{lma:powers} below.
\end{remark}
\begin{remark}Our error bounds also go through with an alternative definition of $\alpha_k$ and $\alpha_{k,g}$ for odd $k$, which is always less than or equal to the one given above and in some cases, strictly smaller. See Remark~\ref{rk:alpha} for this alternative definition. However, the alternative definition is even more cumbersome so we use the one above to simplify the presentation.
\end{remark}
For example, the first four $\alpha_k$'s are given as follows, omitting the argument $r$ for brevity:
\beqs\label{deltas-explicit}
\alpha_1 &= \bar c_3 \\
\alpha_2  &=c_4 ,\\
\alpha_3  &= d^{-1}\bar c_5 ,\\
\alpha_4  &= d^{-1}c_6.
\eeqs
The first several $\alpha_{k,g}$ are given as follows:
\beqs\label{deltas-explicit-f}
\alpha_{0,g}  &=  c_{0,g} ,\\
\alpha_{1,g}  &= d^{-1}\bar c_{1,g} ,\\
\alpha_{2,g}  &= d^{-1} c_{2,g} ,\\
\alpha_{3,g}  &= d^{-2}\bar c_{3,g} ,\\
\alpha_{4,g}  &= d^{-2}c_{4,g} .
\eeqs
\begin{defn}[Bell polynomials]
Let 
\beq\label{def:bell}
B_k(s_1,\dots, s_k) =k! \sum_{\substack{j_1+2j_2+\dots+kj_k=k\\ j_1,\dots,j_k\geq0}}\;\prod_{i=1}^k\frac{s_i^{j_i}}{(i!)^{j_i}(j_i)!}\eeq be the $k$th complete Bell polynomial~\cite[Chapter 12]{Andrews_1984}. For example, 
$$B_0=1,\qquad B_1(s_1)=s_1, \qquad B_2(s_1,s_2)=s_1^2+s_2.$$ 
\end{defn}
The following is a key quantity that will arise in our upper bounds on the terms of the expansion and remainder:
\begin{defn}[Upper bounds on quantities in the expansion]\label{def:calA} For $r=0$ and $r=\Rad$, let
$$\A_k(r) = \sum_{\ell=0}^k\alpha_{k-\ell,g}(r)B_{\ell}(\alpha_1(r),\dots,\alpha_\ell(r)),\qquad k=0,1,2,\dots.
$$Note that $\A_k(0)\leq\A_k(\Rad)$.\end{defn}
To give an example, 
\beqs\label{A2r}
\A_2 &= \alpha_{0,g}(\alpha_1^2+\alpha_2)+\alpha_{1,g}\alpha_1+\alpha_{2,g} = c_{0,g}(\bar c_3^2+c_4) + d^{-1}\bar c_{1,g}\bar c_3 + d^{-1}c_{2,g},
\eeqs where we have omitted the radius argument $r$ for brevity. 

\subsection{Laplace integral expansion}\label{subsec:expand}
The following is the main theorem of this work. Recall that $\epsilon=d/\sqrt n$.
\begin{theorem}\label{thm:expand}Suppose Assumptions~\ref{assume:min},~\ref{assume:CL}, and~\ref{assume:tail} hold for some $\Rad\geq 40$. Then
\beqs\label{eq:thm:expand}
\frac{e^{n\u(\xmin)}\sqrt{\det H}}{(2\pi/n)^{d/2}}\int_{\R^d} g(x)e^{-n\u(x)}dx =g(\xmin )+\sum_{k=1}^{L-1}A_{2k}n^{-k}+\Rem_{L},\eeqs where the $A_{2k}$ are given by the explicit formula in Theorem~\ref{lma:leading}. The remainder $\Rem_{L}$ can be decomposed as $\Rem_{L} =  \kappa_{L}+ \maus{L}+ \maus{\U^c}$, where
\begingroup
\addtolength{\jot}{1em}
\begin{align}
|\kappa_L| &\les_L e^{(\Rad^4c_3(\Rad)^2+c_4(\Rad))\epsilon^2}\A_{2L}(\Rad)\epsilon^{2L},\label{rem-kap}\\
|\maus{L} | &\les_L\max_{0\leq k\leq 2L-1}\A_k(0)\epsilon^k\;e^{-(\Rad-1)^2d/4},\label{rem-tauL}\\
|\maus{\U^c}| &\leq  de^{-\Rad  d/16}.\label{rem-tauU}
\end{align}  \endgroup
Finally, the $A_{2k}$ are bounded as
\beq\label{leading-bound}
n^{-k}|A_{2k}|\les_k \A_{2k}(0)\epsilon^{2k},\qquad\forall k=1,\dots,L-1,
\eeq and have the following representation as a Gaussian expectation:
\beqs\label{A-expect}
A_{2k}=\frac{1}{(2k)!}\sum_{\ell=0}^{2k}{2k\choose\ell}\E\bigg[&\la\nabla^{2k-\ell}f(0), Z^{\otimes 2k-\ell}\ra \\
&\times B_\ell\bigg(\frac{-\lla\nabla^{3}v(0), Z^{\otimes 3}\rra}{2\cdot 3}, \dots,\frac{-\lla\nabla^{\ell+2}v(0), Z^{\otimes \ell+2}\rra}{(\ell+1)(\ell+2)}\bigg)\bigg],\eeqs where $Z\sim\mathcal N(0, I_d)$, $v(x)=\u(\xmin+H^{-1/2}x)-\u(\xmin)$, and $f(x)=g(\xmin +H^{-1/2}x)$.\\
\end{theorem}
\begin{remark}[On the coefficients of the expansion and the bound~\eqref{leading-bound}]\label{rk:coeff}
The terms $A_{2k}n^{-k}$ we obtain are the same as in the classical, fixed $d$ regime. Our representation~\eqref{A-expect} of the coefficients $A_{2k}$ seems to be new, but the explicit formula for the $A_{2k}$ in Theorem~\ref{lma:leading} has already been derived in~\cite{kirwin2010higher}. The bound~\eqref{leading-bound} on the terms of the expansion may be useful for rough approximations of the size of the remainder, since one can write $\Rem_{L}\approx A_{2L}n^{-L}$ and then use~\eqref{leading-bound} with $k=L$. It is easier to compute $\A_{2L}(0)$ then $\A_{2L}(\Rad)$. For the sake of illustration, in Appendix~\ref{app:shun} we give the formula for $A_2$ both as a Gaussian expectation and an explicit sum.
\end{remark}
\paragraph{Discussion.} To gain intuition about our bound on $\Rem_L$, let us omit the exponentially negligible terms $\maus{L},\maus{\U^c}$. Note that this is how~\eqref{RemL-main-bound} in the introduction was obtained, i.e. informally, we have $|\Rem_L|\les\A_{2L}(\u,g)\epsilon^{2L}$ where $\A_{2L}(\u,g):=\exp((\Rad^4c_3(\Rad)^2+c_4(\Rad))\epsilon^2)\A_{2L}(\Rad)$. Now suppose the expression in the exponent is bounded. When $\epsilon\ll1$, boundedness of the exponent translates into a mild assumption on $c_3(\Rad)$ and $c_4(\Rad)$, as well as on $\Rad$ itself. Then essentially Theorem~\ref{thm:expand} gives that $$|\Rem_L|\les\A_{2L}(\Rad)(d^2/n)^L.$$ To further explore the structure of this bound, let us consider the case $L=1$. Using~\eqref{A2r}, we can write $\A_2(\Rad)(d^2/n)$ as follows: 
\beqs\label{Rem1-example}
|\Rem_1|
\les\left\{c_{0,g}(\bar c_3^2+c_4)d^2 + (\bar c_{1,g}\bar c_3 + c_{2,g})d\right\}n^{-1},
\eeqs where we have omitted the radius $R$ argument. The structure of the bound on $\Rem_L$ for higher $L$ is similar. 

If we think of $d$ and $n$ as ``universal" quantities, and of the $c$'s (which derive from $g$ and $\u$) as problem-specific, then we can interpret~\eqref{Rem1-example} as saying that $|\Rem_1|\les d^2/n$ up to problem-specific coefficients. However, we caution against this interpretation because the $c$'s can and often do scale with $d$; for an example, see Section~\ref{sec:log}. Given the fact that the $c$'s do often scale with $d$, the structure~\eqref{Rem1-example} is convenient. Namely, we see that while the $c$'s appearing in front of the highest power of $d$ have the greatest contribution to the remainder, the $c$'s in front of lower powers of $d$ can \emph{afford to be larger without changing the overall magnitude of the bound}. For example, in order for the bound~\eqref{Rem1-example} on $\Rem_1$ to be of size $d^2/n$, it suffices that $c_{0,g},\bar c_3,c_4=\mathcal O(1)$ and $\bar c_{1,g},c_{2,g}=\mathcal O(d)$. (In fact, as noted in Remark~\ref{rk:bar}, we will show that it suffices to bound $c_3$ and $c_{1,g}$ instead of $\bar c_3$ and $\bar c_{1,g}$.) In the aforementioned example from Section~\ref{sec:log}, the derivatives of the function $\u$ ``take full advantage" of this room to grow with $d$. 

In order to determine a general formula for how large each of the $c$'s can be, we use the formulation of the bound in terms of the $\alpha$'s. Indeed, in the bound $|\Rem_L|\les\A_{2L}(\Rad)(d^2/n)^L$, the quantity $\A_{2L}$ (recall Definition~\ref{def:calA}) is a ``pure" $d$-independent polynomial of the $\alpha$'s. Thus if each $\alpha$ is bounded by a constant, then $\A_{2L}(\Rad)$ is bounded by a constant, and hence $\Rem_L$ is bounded by $(d^2/n)^L$. But the $\alpha$'s are given by the $c$'s multiplied by various negative powers of $d$. Thus the $c$'s must be no larger than their corresponding positive powers of $d$ in order for the $\alpha$'s to be bounded by a constant. 

In the next section, we formalize the above discussion to obtain precise conditions under which $|\Rem_L|\les_L (d^2/n)^L$. We can also allow the $c$'s to grow faster, at the cost of obtaining a larger small parameter $\tau^2d^2/n$.

\subsection{Asymptotic expansion under additional assumptions}\label{sec:ae}
We first make a particular choice of the radius $\Rad$:
\beq\label{Rad-simpler}
\Rad =20\l(\frac Ld\log\l(\frac{n}{d^2}\r)\vee 2\r).\eeq It is straightforward to check, using the original bounds~\eqref{rem-tauL},~\eqref{rem-tauU} on $\maus{L}$ and $\maus{\U^c}$, that we have the following further bound with this choice of $\Rad$:
\beq\label{prop:R1}
|\maus{L}|+|\maus{\U^c}|\les_L \l(1\vee\max_{0\leq k\leq 2L-1}\A_k(0)\epsilon^{k}\r)\epsilon^{4L}.
\eeq
Recall that $\Rad$ still appears in the bound~\eqref{rem-kap} on $\kappa_L$, via the exponential factor $\exp([\Rad^4c_3(\Rad)^2+c_4(\Rad)]\epsilon^2)$ and $\A_{2L}(\Rad)$. We show below that for $\Rad$ as in~\eqref{Rad-simpler}, the exponential factor can be bounded by a constant under mild conditions. The quantity $\A_{2L}(\Rad)$ is given in terms of $c_k(\Rad),c_{k,g}(\Rad)$, which are given by suprema of derivative operator norms in a neighborhood of $\xmin$ of radius $\Rad\sqrt{d/n}$. To show that the choice~\eqref{Rad-simpler} does not make $\A_{2L}(\Rad)$ too large, we check that $\Rad\sqrt{d/n}\ll 1$, which implies that we need to take the supremum only over a vanishingly small neighborhood of $\xmin$.  It is straightforward to check that this is indeed the case as long as $d/n\ll1$ and $n\gg1$. However, if the derivatives of $\u$ or $g$ grow extremely rapidly as $x$ moves away from $\xmin$, then a different choice of $\Rad$ may be warranted.\\

We now let $\tau>0$ be some possibly large parameter. The following lemma states conditions on the derivatives of $g$ and $\u$ which will allow us to show that~\eqref{eq:thm:expand} is an asymptotic expansion in powers of the small parameter $(\tau\epsilon)^2$.
\begin{lemma}\label{lma:powers}Let $\tau>0$ satisfy
\beq\label{tau-ep}
\tau\epsilon \leq \frac{d^2}{\log^2(n/d^2)}\wedge 1.
\eeq  
Suppose
\begin{align}
|c_{k,g}(0)|&\les_k d^{\frac {k+1}2}\tau^k\qquad\forall k=1,3,\dots,2L-1,\label{ckg-0}\\ 
 |c_{k}(0)|&\les_k d^{\frac {k+1}2 - 2}\tau^{k-2}\qquad\forall k=3,5,\dots,2L+1,\label{ck-0}\\
|c_{k,g}(\Rad)|&\les_{L}d^{\frac k2}\tau^k\qquad\forall k=0,2,\dots,2L,\label{ckg-ev-R}\\
 |c_{k}(\Rad)|&\les_{L}d^{\frac k2 -2 }\tau^{k-2}\qquad\forall k=4,6,\dots, 2L+2,\label{ck-ev-R} 
  \end{align} where $\Rad$ is as in~\eqref{Rad-simpler}. Then
  \begingroup 
\addtolength{\jot}{1em}
\beqs\label{Ak-rad}
&\A_{j}(0)\leq\A_{j}(\Rad)\les_{L}\tau^{j}\quad\forall 0\leq j\leq 2L,\\
&(\Rad^4c_3(\Rad)^2+c_4(\Rad))\epsilon^2\les_L 1.\eeqs
\endgroup
\end{lemma}
\begin{remark}
In~\eqref{ckg-ev-R} and~\eqref{ck-ev-R}, we have allowed the bounds on $c_{k,g}(\Rad)$ and $c_k(\Rad)$ to depend on $L$ since the radius $\Rad$ itself depends on $L$. 
\end{remark}
\begin{remark}
We actually need to bound $\bar c_k(\Rad)=c_k(0)+\epsilon c_{k+1}(\Rad)$ and $\bar c_{k,g}(\Rad)=c_{k,g}(0)+\epsilon c_{k+1,g}(\Rad)$ for odd $k$ in order to bound the $\A$'s. But it is straightforward to show that if $k$ is odd, then we can combine~\eqref{ck-0} with~\eqref{ck-ev-R} to obtain a bound on $\bar c_k(\Rad)$ of the exact same order as~\eqref{ck-0}. The same is true for $\bar c_{k,g}(\Rad)$.
\end{remark}
See Appendix~\ref{app:main} for the proof, which relies on the following key property: $$\A_k\tau^{-k} = \sum_{\ell=0}^{k}\alpha_{k-\ell,g}\tau^{-(k-\ell)}B_{\ell}(\alpha_1\tau^{-1},\dots,\alpha_\ell\tau^{-\ell}).$$ This structure shows that it suffices to bound $\alpha_k\tau^{-k}$ and $\alpha_{k,g}\tau^{-k}$ to ensure $\A_k\tau^{-k}$ remains bounded. Example~\ref{ex:quart} and Section~\ref{sec:log} below give examples of functions $\u$ for which the conditions from Lemma~\ref{lma:powers} are satisfied with $\tau=1$. \\

In preparation for the following corollary, recall that a generalized asymptotic expansion with respect to the small parameter $(\tau\epsilon)^2$ means the following~\cite[Chapter 1]{wongbook}: the remainder satisfies $\Rem_L=\mathcal O((\tau\epsilon)^{2L})$ as $\tau\epsilon\to0$ for each $L=1,2,3,\dots$.

\begin{corollary}[Asymptotic expansion in powers of $(\tau\epsilon)^2$]\label{prop:R1:corr}
Suppose Assumptions~\ref{assume:min},~\ref{assume:CL}, and~\ref{assume:tail} hold with $\Rad$ as in~\eqref{Rad-simpler}, and let $\tau$ satisfy~\eqref{tau-ep}. If~\eqref{ckg-0}-\eqref{ck-ev-R} hold, then
\begin{align}
e^{n\u(\xmin)}&\frac{\sqrt{\det (nH)}}{(2\pi)^{d/2}}\int_{\R^d} g(x)e^{-n\u(x)}dx =g(\xmin )+\sum_{k=1}^{L-1}A_{2k}n^{-k}+\Rem_{L},\label{expand1}\\
&|\Rem_{L}| \les_{L} \l(\tau\epsilon\r)^{2L}+\epsilon^{4L}.\label{expand3}
\end{align} Therefore, if $\epsilon\leq \tau$, then~\eqref{expand1}-\eqref{expand3} is a generalized asymptotic expansion in powers of $(\tau\epsilon)^2$.
\end{corollary}
The proof is immediate by combining Theorem~\ref{thm:expand},~\eqref{prop:R1}, and Lemma~\ref{lma:powers}. \\

The most classical kind of asymptotic expansion is of Poincaré type and in power series form~\cite[Chapter 1]{wongbook}. This means the remainders $\Rem_L$ scale as powers of the small parameter, as opposed to only having \emph{upper bounds} which scale in powers of the small parameter, as in a generalized expansion. Determining whether the expansion is indeed ``classical" in this sense requires studying the terms $A_{2k}n^{-k}$ to show they satisfy $A_{2k}n^{-k} \asymp (\tau\epsilon)^{2k}$. This may be possible to do on a case-by-case basis; in particular, we do so in Example~\ref{ex:quart} below.


\begin{example}[Case $\tau=1,L=1,g\equiv1$]\label{ex:L1} We consider this simple setting in order to make a comparison with the prior work~\cite{tang2021laplace}. To further simplify things, assume $H=I_d$ and $n \leq d^2e^d$, in which case~\eqref{Rad-simpler} gives $\Rad=40$. In addition to regularity and growth at infinity, the main assumptions of Corollary~\ref{prop:R1:corr} are that
\beqs\label{u34}
\|\nabla^3\u(\xmin)\|&\leq C, \\
\|\nabla^4\u(x)\| &\leq C\quad\forall \|x-\xmin\|\leq 40\sqrt{d/n}.
\eeqs
Under these conditions, we obtain that $|\Rem_1|\les d^2/n$. 

Meanwhile,~\cite{tang2021laplace} assumes some integrability of $\int e^{-n\u(x)}dx$ in the complement of a small ball around $\xmin$, which is analogous to our condition on growth of $\u$ at infinity. Conditions on the third and fourth derivatives in~\cite{tang2021laplace} are also given, which lead to different bounds depending on how large the derivative norms are. The closest analogue to~\eqref{u34} is to set $c_3=c_4=1$ in Assumptions 3 and 4 of~\cite{tang2021laplace}, which gives rise to the condition
\beq
\|\partial_{x_i}\nabla^2\u(x)\|\leq C,\quad \|\partial_{x_i}\partial_{x_j}\nabla^2\u(x)\|\leq C\eeq for all $i,j=1,\dots,d$ and for all $x$ in a ball around $\xmin$. This is very similar to ~\eqref{u34}. Finally, when $H=I_d$, the condition on $H$ from Assumption 2 holds with $c_\infty=0$. Under these conditions, the authors obtain in Theorem 3.1 that $|\Rem_1|\les d^3/n$.

Thus under very similar conditions to those in prior work, we improve the bound on $\Rem_1$ from $d^3/n$ to $d^2/n$. The next example shows that the improved bound is tight.\end{example}

\begin{example}[Quartic exponential]\label{ex:quart}In this example, we demonstrate functions $g,\u$ satisfying the conditions of Corollary~\ref{prop:R1:corr}, for which the obtained bound~\eqref{expand3} is tight. Namely, let
\beq\label{g-z-quart}g(x)=1,\qquad \u(x)=\frac12\|x\|^2+\frac{1}{24}\|x\|^4.\eeq The Laplace expansion of $\int ge^{-n\u}$ takes the following form (see Appendix~\ref{app:main}):
\begingroup
\addtolength{\jot}{0.5em}
\beqs\label{quart-decomp}
\int_{\R^d}e^{-n\l(\frac{\|x\|^2}{2}+\frac{\|x\|^4}{24}\r)}dx &= \l(\frac{2\pi}{n}\r)^{d/2}\l(1 + \sum_{k=1}^{L-1}A_{2k}n^{-k}+\Rem_L\r), \\
A_{2k}&=\l(\frac{-1}{6}\r)^k\frac{\Gamma(2k+d/2)}{k!\Gamma(d/2)}.\eeqs \endgroup Furthermore, it is straightforward to show from the above formula for the terms that they satisfy
\beq\label{Ad2k}A_{2k}n^{-k}\asymp_k (d^2/n)^k.\eeq
We now state the bounds implied by Corollary~\ref{prop:R1:corr}.
\begin{prop}\label{prop:quart}Suppose $d^2/n\leq 1/4$. Then the functions $g,\u$ in~\eqref{g-z-quart} satisfy all the conditions from Corollary~\ref{prop:R1:corr} with $\tau=1$, and hence the corollary implies
\beq\label{corrshows}|\Rem_L|\les_L (d^2/n)^L.\eeq Furthermore, fix any $L\geq1$ and suppose $d^2/n$ is smaller than some constant depending on $L$. Then in fact,
\beq\label{RemLdn}\Rem_L\asymp (d^2/n)^L.\eeq
\end{prop} 
See Appendix~\ref{app:main} for the short proof.
\begin{remark}~\eqref{Ad2k} and~\eqref{RemLdn} show that the expansion~\eqref{quart-decomp} is a classical asymptotic expansion in powers of $\epsilon^2=d^2/n$. The error after truncating to order $L-1$ is precisely of the order of the $L$th term.
\end{remark}
The above remark is itself an interesting fact about the expansion~\eqref{quart-decomp}, but our main goal in presenting Proposition~\ref{prop:quart} is to show that Corollary~\ref{prop:R1:corr} gives tight bounds. Specifically, we now know from~\eqref{RemLdn} that the true values of the remainder are constant multiples of powers of $d^2/n$. We conclude that \emph{the bound~\eqref{corrshows} supplied by Corollary~\ref{prop:R1:corr} is tight}.

Interestingly, showing $\Rem_L\asymp (d^2/n)^L$ hinges on the bound~\eqref{corrshows}. Indeed,~\eqref{corrshows} and~\eqref{Ad2k} give $|A_{2L}n^{-L}|\gtrsim_L (d^2/n)^L$ and $|\Rem_{L+1}|\les_{L}(d^2/n)^{L+1}$, so that $|\Rem_L|\gtrsim_L (d^2/n)^L$ if $d^2/n$ is small enough. Thus we used the bound in Corollary~\ref{prop:R1:corr} to help prove that this very bound is tight.



\end{example}

\section{An example from statistics}\label{sec:log}
In this section, we study the high-dimensional Laplace expansion of $\int ge^{-n\u}$ for a function $\u$ which is both random and depends itself on $n$. The setting we study here arises e.g. in Bayesian inference in statistics. In this context, one is interested in computing summary statistics~\cite[Chapter 2.3]{bda3} such as the mean and variance of the \emph{posterior probability density}, which takes the form $\pi(x)\propto e^{-n\u(x)}$~\cite{schillings2020convergence}. For example, the mean of the posterior is given by
$$\int_{\R^d}x\pi(x)dx = \frac{\int_{\R^d}xe^{-n\u(x)}dx}{\int_{\R^d}e^{-n\u(x)}dx}.$$ Thus computing such summary statistics reduces to computing a ratio $\int ge^{-n\u}/\int e^{-n\u}$ of two Laplace-type integrals, where $g$ is determined by the statistic of interest. Meanwhile, $\u$ is determined by the statistical model and the observed data, and this function is random due its dependence on the data. The normalizing constant $\int_{\R^d}e^{-n\u(x)}dx$ is also itself of interest, in Bayesian model selection~\cite[Chapter 5]{ando2010bayesian}. 

In the below analysis, we consider a particular widely-used statistical model, which thereby determines the function $\u$, and verify that this $\u$ satisfies our assumptions. We leave $g$ unspecified, since it depends on the summary statistic of interest. In any case, $g$ is typically a simple function, such as a polynomial, for which verifying the assumptions is trivial. 

The purpose of this example is threefold. First, it allows us to demonstrate how our assumptions can be checked in a nontrivial example. Second, the example is directly relevant to Bayesian computation. Third, the example demonstrates that derivative tensor norms can grow with $d$ and more importantly, that this growth with $d$ does \emph{not} affect the definition of the small parameter, thanks to the slack in our bound on $\Rem_L$ (recall the discussion below Remark~\ref{rk:coeff}). We will see that nearly all of the slack gets used up, i.e. our bounds on the derivative tensor norms of $\u$ are nearly at the limit of what is allowed by~\eqref{ck-0} and~\eqref{ck-ev-R}.



The analysis in this section follows along similar lines as in our previous works~\cite{katsBVM,katskew}.
\subsection{Definition of $\u$}

To define $\u$, we first introduce a function $\phi:\R\to\R$ such that 
\beq\label{letphi}
\phi\;\text{strictly convex}, \quad \phi\in C^\infty(\R),\quad\|\phi^{(k)}\|_\infty=\sup_{t\in\R}|\phi^{(k)}(t)|<\infty\quad\forall k=3,4,5,\dots.\eeq Next, let $X_i\iid\mathcal N(0, I_d)$, $i=1,\dots,n$, and define the function $\psi:\R^d\to\R$ according to
$$\psi(x)=\frac1n\sum_{i=1}^n\phi(X_i^\T x).$$ It is straightforward to show that $\psi$ is strictly convex whenever the $X_i$ span $\R^d$, and this occurs with probability 1. 
Finally, for a point $\xmin \in\R^d$ such that $\|\xmin \|=1$, let
\beqs\label{udef-log}
\u(x) &= \psi(x) - \nabla\psi(\xmin )^\T x\\
&= \frac1n\sum_{i=1}^n\phi(X_i^\T x) - \frac1n\sum_{i=1}^n\phi'(X_i^\T \xmin )X_i^\T x.
\eeqs We assume $\|\xmin \|=1$ for simplicity, but any $\xmin $ whose norm is bounded independently of $n$ and $d$ would suffice for our purposes. 
Functions $\u(x)$ of the form~\eqref{udef-log} arise in the context of generalized linear models (GLMs) in statistics~\cite{agresti2015foundations}. In particular, if $\phi(t)=\log(1+e^t)$, then $\u(x)$ is the negative population log likelihood for the unknown coefficient vector $x$ in a logistic regression model. 

 \subsection{Laplace expansion}\label{subsec:lap:log}
Throughout this section we treat $\phi$ as fixed and do not indicate the dependence of various quantities on $\phi$. As discussed above, we will not choose a particular function $g$, so the main work in this section is to check the assumptions of Corollary~\ref{prop:R1:corr} for the function $\u$ from~\eqref{udef-log}. We note that it is meaningless to check the assumptions on $\u$ for a particular draw of the normally distributed data $X_1,\dots, X_n$, since the support of the Gaussian is all of $\R^d$. Thus the data can be any set of $n$ points in $\R^d$, and so the derivative tensor norms of $\u$ can be arbitrarily large. Instead, we will make statements that hold with high probability over possible draws of the data.

First, note that $\u$ is strictly convex with probability 1 because so is $\psi$, as discussed above. Moreover, it is clear that $\u$ has a unique strict global minimizer $x=\xmin $ (this is true for any $\u$ of the form $\u(x)= \psi(x) - \nabla\psi(\xmin )^\T x$, for a strictly convex $\psi$). Therefore, $\u$ satisfies Assumption~\ref{assume:min}. Next, since $\phi\in C^{\infty}(\R)$, we have $\u\in C^\infty(\R^d)$, so the function $\u$ satisfies the regularity Assumption~\ref{assume:CL} regardless of how the radius $\Rad$ is chosen. Next, we study the Hessian at $\xmin$. We have 
\beq\label{H-log}H=\nabla^2\u(\xmin ) = \frac1n\sum_{i=1}^n\phi''(X_i^\T \xmin )X_iX_i^\T.\eeq We have the following lower bound on $H$.
\begin{lemma}\label{lma:Hlb} Suppose $d/n<1/2$. Then there are absolute constants $C,\lambda>0$ such that the event
\beq
E_H=\l\{\frac1n\sum_{i=1}^n\phi''(X_i^\T \xmin )X_iX_i^\T \succ \lambda I_d\r\}\eeq has probability at least $1-4e^{-Cn}$.
\end{lemma}
This follows immediately from~\cite[Lemma 4, Section 4]{surMLE1}. The constant $\lambda$ depends on the smallest value of $\phi''$ in some bounded interval whose width depends on $\|\xmin \|=1$. This is where the boundedness of $\xmin $ is used. That $\|\xmin\|=1$ and $\phi$ is fixed allows us to treat $\lambda$ as an absolute constant, so we omit the dependence of further constants on $\lambda$. 

It remains to bound the $c_k$'s and check Assumption~\ref{assume:tail}. We start with the bound on the $c_k$'s, since we then use this bound for $k=3$ to help check Assumption~\ref{assume:tail}. Now, thanks to Lemma~\ref{lma:Hlb}, it suffices to bound the \emph{unweighted} operator norms, denoted $\|\nabla^k\u\|$. This will translate into bounds on $\|\nabla^k\u\|_H$, i.e. on the $c_k$'s. By direct calculation, we obtain the following expression for the unweighted operator norms:
\beq\label{uderivlog}
\|\nabla^k\u(x)\|= \sup_{\|u\|=1}\frac1n\left|\sum_{i=1}^n\phi^{(k)}(X_i^\T x)(X_i^\T u)^k\right|.
\eeq
The following lemma bounds~\eqref{uderivlog} with high probability.
\begin{lemma}\label{lma:adam}
Suppose $d\leq n$. Then the event
$$E_k=\l\{\sup_{x\in\R^d}\|\nabla^k\u(x)\|\leq C_k\l(1+\frac{d^{k/2}}{n}\r)\r\}$$ holds with probability at least $1-4\exp(-C_k'\sqrt n)$, where $C_k,C_k'$ are constants depending only on $k$.
\end{lemma}See Appendix~\ref{app:log} for the proof, which almost immediately follows from~\cite[Proposition 4.4]{adamczak2010quantitative}. We can now combine the above two lemmas, and use the inequality $\|\nabla^k\u\|_H\leq\lambda^{-k/2}\|\nabla^k\u\|$. This immediately yields the following corollary.
\begin{corollary}\label{corr:c:log}
The event $E:=E_H\cap E_3\cap E_4\cap \dots \cap E_{2L+2}$ has probability at least $1-C_Le^{-C_L'\sqrt n}$ for some constants $C_L,C_L'>0$ depending only on $L$. Furthermore, for some constants $C_k$, the following bounds hold on the event $E$:
\beqs\label{ck-log}
\sup_{r\geq0}c_k(r)&\leq C_{k}\bigg(1+\frac{{\sqrt d}^k}{n}\bigg),\quad\forall \, 3\leq k\leq 2L+2.
\eeqs 
\end{corollary} 
With this corollary in hand, we now claim the key growth conditions on the $c_k$ from Corollary~\ref{prop:R1:corr}, namely~\eqref{ck-0} and~\eqref{ck-ev-R}, are satisfied with $\tau=1$. To see how much ``room" we have to satisfy these conditions, write~\eqref{ck-log} as
\beq\label{ck-log-2}\sup_{r\geq0}c_k(r)\les_k 1+d^{\frac k2 - 2}\,\frac{d^2}{n}.\eeq Since our goal is to obtain an expansion in powers of $d^2/n$, we are free to assume $d^2\leq n$. Using this assumption, we see from~\eqref{ck-log-2} that for even $k$, we obtain precisely the scaling $d^{\frac k2-2}$ required in~\eqref{ck-ev-R}. For $k=3$, we also precisely obtain $c_3(0)\les 1=d^{\frac 42-2}$ from~\eqref{ck-log-2}, as required in~\eqref{ck-0}. For odd $k\geq5$,~\eqref{ck-log-2} yields $c_k(r)\les_k d^{\frac k2-2}$, which is a power of $\sqrt d$ below the threshold $d^{\frac{k+1}{2}-2}$ required in~\eqref{ck-0}. Therefore, \emph{for all but odd $k\geq5$, the bound~\eqref{ck-log-2} on the growth of the derivatives with $d$ in this example are exactly at the threshold allowed by the growth conditions~\eqref{ck-0} and~\eqref{ck-ev-R}.}\\

It remains to check condition~\eqref{u-imply} from Assumption~\ref{assume:tail}. 
\begin{lemma}\label{lma:log:c3}Suppose $d^2\leq n$ and $n/d\geq C$ for some $C$ large enough. Then on the event $E_3\cap E_H$, the condition~\eqref{u-imply} holds for all $\|x-\xmin\|_H\geq\sqrt{d/n}$.
\end{lemma} 
\begin{proof}Since $\u$ is convex, we can apply Lemma~\ref{lma:cvx} with $r=1$. The lemma implies that if $c_3(1)\sqrt{d/n}\leq 1$, then~\eqref{u-imply} is satisfied. But indeed, by Corollary~\ref{corr:c:log} and the assumption $d^2\leq n$, we have $c_3(1)\leq C$ on $E_3\cap E_H$. Thus $c_3(1)\sqrt{d/n}\leq 1$ for $n/d$ large enough.
\end{proof} 
\begin{corollary}\label{corr:expand:log}
Let $\u$ be as in~\eqref{udef-log}, with $\|\xmin \|=1$, $\phi$ as in~\eqref{letphi}, and $X_i\iid\mathcal N(0, I_d)$, $i=1,\dots, n$. Let $$\Rad = 20\l(\frac Ld\log\l(\frac{n}{d^2}\r)\vee 2\r).$$ Suppose $g$ is $2L$ times continuously differentiable in the region $\{\|x-\xmin \|_H \leq \Rad\sqrt{d/n}\}$ and $|g(x)| \leq \exp(\sqrt {dn}\|x-\xmin\|_H/12)$ for all $\|x-\xmin\|_H\geq\Rad\sqrt{d/n}$. Also, suppose there exist constants $C_k>0$ such that
\beqs\label{gcond-log}
\|\nabla^kg(\xmin)\|&\leq C_kd^{\frac {k+1}2},\qquad\forall \; k=1,3,\dots,2L-1,\\
\sup_{\|y\|_H\leq\Rad\sqrt{d/n}}\|\nabla^kg(\xmin+y)\|&\leq C_{L} d^{\frac k2},\qquad\forall \; k=0,2,\dots,2L.
\eeqs
If $d^2/n\leq1/4$ and $n/d > C$ for $C>0$ large enough, then on the high probability event $E$ from Corollary~\ref{corr:c:log} it holds
\beqs\label{RemL-log}
e^{n\u(0)}\frac{\sqrt{\det (nH)}}{(2\pi)^{d/2}}\int_{\R^d} g(x)e^{-n\u(x)}dx &=g(\xmin )+\sum_{k=1}^{L-1}A_{2k}n^{-k}+\Rem_{L},\\
|\Rem_{L}| &\les_{L}(d^2/n)^L.
\eeqs
\end{corollary}The proof follows immediately from the above lemmas and Corollary~\ref{prop:R1:corr}. In particular, $d^2/n\leq1/4$ ensures~\eqref{tau-ep} is satisfied. Note that we removed the $H$ weighting from the derivative tensor norms in~\eqref{gcond-log}. This is possible since we know that $H\succeq\lambda I_d$ on the event $E$, for an absolute constant $\lambda>0$.

\begin{remark}The focus of the present paper is on the remainder bounds rather than on the terms of the expansion. However, for this example, the terms are of practical interest e.g. in Bayesian inference, where $\int ge^{-n\u}$ represents an unnormalized posterior integral in a GLM. We therefore also present the formula for $A_2$; see Appendix~\ref{app:log}. 

The function $\u$ considered in~\eqref{udef-log} has two simplifications compared to the ``actual" $\u$ used in practice: $-n\u$ is a population log likelihood, whereas typically $-n\u$ is a sample log likelihood plus a log prior. A prior can easily be incorporated into the function $g$; in other words, instead of defining $e^{-n\u}$ to be the unnormalized posterior density, as we did in the introduction to this section, we can define it to be the likelihood. The fact that $-n\u$ is a population- rather than sample log likelihood has no bearing on the formula for the terms, since second and higher-order derivatives of the sample and population log likelihood are the same for GLMs in canonical form. \end{remark}

\section{Proof of Theorem~\ref{thm:expand}}\label{sec:proof}In Section~\ref{sec:initsimp}, we make some initial simplifications to the integral of interest. In Section~\ref{sec:outline} we give a proof outline and compare our method to other techniques to derive the Laplace expansion. In Section~\ref{sec:taylor}, we derive a Taylor expansion of our integral in a local region, with respect to the small parameter $1/\sqrt n$. This leads to the decomposition~\eqref{eq:thm:expand}. We bound the terms $A_{2k}$ and the remainder $\Rem_L$ in Section~\ref{sec:rem}, and obtain the explicit formula for the $A_{2k}$ in Section~\ref{sec:lead}.

\subsection{Initial simplifications}\label{sec:initsimp}
We begin with a reparameterization: 
$$
e^{n\u(\xmin)}\sqrt{\det H}\int g(y)e^{-n\u(y)}dy = \int f(x)e^{-nv(x)}dx,
$$ where
$$
v(x)=\u(\xmin +H^{-1/2}x)-\u(\xmin),\qquad f(x)=g(\xmin +H^{-1/2}x).
$$ 
We record Assumptions~\ref{assume:min},~\ref{assume:CL}, and~\ref{assume:tail} as well as the quantities $c_k,c_{k,g}$ in the language of $v$ and $f$. First, $v$ has a unique global minimizer at $x=0$, with $\nabla^2v(0)=I_d$. Second,
\beqs\label{fv-reg}
f\in C^{2L}(n^{-1/2}\bar\U),\qquad v\in C^{2L+2}(n^{-1/2}\bar\U),
\eeqs where $\bar\U=\{x\in\R^d:\|x\|\leq\Rad\sqrt d\}$ is the closure of the open set $\U$. As explained in Section~\ref{subsec:assume}, this condition means that the derivatives $\nabla^kf$, $k=0,\dots,2L$ and $\nabla^kv$, $k=0,\dots,2L+2$ exist and are uniformly continuous in the open set $n^{-1/2}\U$, and can therefore be uniquely extended to continuous functions up to the boundary. Third, it holds
\begin{align}
|f(x)|\exp(-nv(x)) \leq \exp(-\sqrt{dn}\|x\|/4),\quad\forall \|x\|\geq\Rad\sqrt{d/n}.\label{v-av-growth-ii}\end{align}
Finally, note that $\|\nabla^kv(x)\| = \|\nabla^k\u(\xmin +H^{-1/2}x)\|_{H}$ for all $k\geq1$ and $\|\nabla^kf(x)\| = \|\nabla^kg(\xmin +H^{-1/2}x)\|_{H}$ for all $k\geq0$. Using this it is straightforward to show that $c_k,c_{k,g}$ from Definition~\ref{def:cvf} can be expressed as bounds on the unweighted operator norms of the tensors $\nabla^kv$ and $\nabla^kf$. Specifically,
\beqsn
\sup_{\|x\|\leq \Rad\sqrt{d/n}}\|\nabla^kv(x)\|=&c_{k}(\Rad),\\
\sup_{\|x\|\leq \Rad\sqrt{d/n}}\|\nabla^kf(x)\|=&c_{k,g}(\Rad).\eeqsn
Thus although the notation $c_{k,g}$ refers to the function $g$, one should keep in mind that $c_{k,g}$ is also an upper bound on $\|\nabla^kf\|$. Correspondingly $\alpha_{k,g}$ is a downweighted upper bound on $\|\nabla^kf\|$.\\

Now, define the functions $F,\rr:\R\times\R^d\to\R$, and $I:\R\to\R$ as follows:
\begingroup
\addtolength{\jot}{0.5em}
\begin{align}
F(t,x) &= f(tx),\label{F-def}\\
\rr(t,x)&=\begin{dcases}
\frac{v(tx)}{t^2}-\frac12\|x\|^2,&t\neq0,\\
0, &t=0,\end{dcases}\label{r-def}\\
I(t)&=\int_\U F(t, x)\,e^{-\rr(t, x)}\gamma(dx),\label{I-def}
\end{align}
\endgroup 
where $\gamma(dx)=(2\pi)^{-d/2}e^{-\|x\|^2/2}dx$ is the standard normal distribution and $$\U=\{x\in\R^d:\|x\|<\Rad\sqrt d\}.$$ Using basic manipulations, we obtain the following initial decomposition of the integral of interest. 
\begin{lemma}\label{lma:basic-decomp}It holds

\begingroup
\addtolength{\jot}{0.5em}
\beqs\label{main-ratio}
\frac{e^{n\u(\xmin)}\sqrt{\det H}}{(2\pi/n)^{d/2}}&\int_{\R^d} g(x)e^{-n\u(x)}dx\\
= &\int_\U F(1/\sqrt n, \,x)\,e^{-\rr(1/\sqrt n, x)}\gamma(dx)+(2\pi)^{-d/2}\int_{\U^c}f(x/\sqrt n)e^{-nv(x/\sqrt n)}dx\\
=&\,\;\quad I(1/\sqrt n) \qquad\qquad\qquad\qquad\;\;+\qquad\qquad\qquad \maus{\U^c},
\eeqs
\endgroup
 where $\maus{\U^c}$ is defined to be the second summand in the second line.
\end{lemma} See Appendix~\ref{app:sec:proof} for the proof. 

\subsection{Proof outline and comparison to the literature}\label{sec:outline}
The local integral $I(1/\sqrt n)$ is the main object of study. We analyze it in two steps. First, we Taylor expand $I(t)$ about $0$ to order $L-1$, with $L$th order remainder. Substituting $t=1/\sqrt n$, we already essentially obtain the decomposition~\eqref{eq:thm:expand}, with $A_{2k}$ as in~\eqref{A-expect}, and with the remainder $\Rem_L$ given by a nearly explicit integral. See Lemma~\ref{lma:Taylor} in Section~\ref{sec:taylor} for this main first result. The second step is to bound the $A_{2k}$ and $\Rem_L$, which we do in Section~\ref{sec:rem}. This second step is where the high-dimensionality necessitates a much more careful analysis. 

There is a final less central part of the proof, which is thematically unrelated to the two steps we have described. Namely, we obtain an explicit formula for the $A_{2k}$ and confirm that it matches the known formula for the terms of the standard multivariate Laplace asymptotic expansion. This step confirms that the expansion we have derived (and justified in the high-dimensional regime) is precisely the standard expansion known to hold in fixed $d$. The explicit formula obtained by~\cite{kirwin2010higher} follows from a different intermediate representation of the $A_{2k}$, which is why we need to do this final step. See Section~\ref{sec:lead} for this calculation, which simply requires careful tracking of indices and tensor manipulations.

We now describe the two main steps in more detail. We then compare our proof technique to those of prior works. 


\paragraph{Step 1: Expansion of $I(1/\sqrt n)$.} Note that $\rr(1/\sqrt n,x) = nv(x/\sqrt n)-\|x\|^2/2$ is the remainder after removing the first, quadratic term in the Taylor expansion of $x\mapsto nv(x/\sqrt n)$ about zero. Thus the Taylor series for $\rr(1/\sqrt n,x)$ is given by the higher order terms in the Taylor series of $nv(x/\sqrt n)$. Substituting this series into the argument of the exponential $\exp(-\rr)$, and replacing $f(x/\sqrt n)$ by its Taylor series expansion about zero, 
we get the following formal representation of the integrand of $I(1/\sqrt n)$:
\beqs\label{F-r-exp}
F(1/\sqrt n, x)&e^{-\rr(1/\sqrt n,\, x)} = f(x/\sqrt n)\exp(-[nv(x/\sqrt n)-\|x\|^2/2])\\
= &\l(f(0)+n^{-1/2}\la\nabla f(0), x\ra + \frac12n^{-1}\la\nabla^2f(0), x^{\otimes 2}\ra+\dots\r)\\
&\times\exp\l(-\l[\frac{1}{3!}n^{-1/2}\la\nabla^3v(0), x^{\otimes 3}\ra+\frac{1}{4!}n^{-1}\la\nabla^4v(0), x^{\otimes 4}\ra+\dots\r]\r)
\eeqs 
Thus we need to expand the integral of this expression in powers of $n^{-1/2}$. We do this by Taylor expanding $I(t)=\int_{\U}F(t,\cdot)e^{-\rr(t,\cdot)}d\gamma$ about $t=0$ and evaluating the Taylor expansion at $t=n^{-1/2}$. Note that the coefficients of this expansion are indeed given by Gaussian expectations, of the form
$$
\int_{\U}\partial_t^k\l(F(t,\cdot)e^{-\rr(t,\cdot)}\r)\big\vert_{t=0}d\gamma.
$$
The derivatives at $t=0$ reduce to polynomials in $x$. The computation of the $t$ derivatives of $Fe^{-\rr}$ is simplified by the use of Faa di Bruno's formula for derivatives of composite functions, such as $\exp(-\rr(t,x))$. The utility of Faa di Bruno's formula in deriving the Laplace asymptotic expansion has been discussed in~\cite{wojdylo2006computing}, in the context of a different proof technique.

\paragraph{Step 2. Bounding $A_{2k}$ and $\Rem_L$.}In this overview, we will highlight two key intermediate bounds at the heart of our proof and which yield a tighter overall dimension dependence than has been obtained in prior works in high dimensions (see Section~\ref{intro:stateofart}). See Section~\ref{sec:rem} for how these key intermediate bounds are used. 

Specifically, expectations of polynomials and exponential functions of $Z\sim\mathcal N(0,I_d)$ arise as a central quantity, and we bound them using the following two powerful inequalities:
 \begin{align}
\E\l[\lla T, Z^{\otimes k}\rra^q\r]^{\frac1q} &\les_{k,q}  \|T\|{\sqrt d}^{k-1}, \qquad\text{$k$ odd}\label{intro:lip1}\\
\E\l[e^{q\phi(Z)}\ind_\U(Z)\r]^{\frac1q} &\leq \exp\l(\E[\phi(Z)\ind_\U(Z)]+\frac{q}{2}\sup_{x\in\U}\|\nabla \phi(x)\|^2\r).\label{intro:lip2}
\end{align} Here, $q$ is even, $T$ is a tensor of appropriate dimensions and $Z\sim\mathcal N(0, I_d)$. The bound~\eqref{intro:lip2} is an immediate application of the concentration of Lipschitz functions of multivariate Gaussians~\cite[Chapter 2]{pisier1986probabilistic},~\cite[Chapter 5]{vershynin2018high}; see Section~\ref{app:sec:lip}. We prove the bound~\eqref{intro:lip1} in Proposition~\ref{lma:pk-bds} using this same concentration property and a bit of extra work. We note that~\eqref{intro:lip1} may also follow as a special case of the work~\cite{latala2006estimates} on moments of decoupled Gaussian chaoses, when combined with the result of~\cite{de1995decoupling} showing that the decoupling does not significantly affect the moments. However, we choose to provide our own proof to make the paper self-contained.

The bound~\eqref{intro:lip1} for odd $k$ is tighter than the bound $\E\l[\lla T, Z^{\otimes k}\rra^q\r]^{\frac1q}\les_{k,q}\|T\|{\sqrt d}^k$ which holds for all $k$ and which follows by the simple operator norm inequality, $|\la T, Z^{\otimes k}\ra|\leq\|T\|\|Z\|^k$ (for even $k$, this is the best we can do). Let us show why~\eqref{intro:lip2} also yields a tighter bound than the more direct bound $\E\l[e^{q\phi(Z)}\ind_\U(Z)\r]^{\frac1q}\leq \exp(\sup_{x\in\U}\phi(x))$. To see this we note that the function relevant to us is $$\phi(x)=\rr(t,x)\approx  \frac{t}{3!}\la\nabla^3v(0), x^{\otimes 3}\ra$$ for some $t\in(0,1/\sqrt n)$, where the approximation follows from a Taylor expansion about $x=0$. Thus to simplify the discussion let us redefine $\phi$ to be $\phi(x)=n^{-1/2}\la T, x^{\otimes3}\ra$. For this $\phi$, the supremum of $\|\nabla\phi(x)\|^2$ over all $\|x\|\leq\Rad\sqrt d$ is given by $\mathcal O((d/\sqrt n)^2)$, up to the dependence on $\Rad$ and $\|T\|$. In contrast, the supremum of $\phi(x)$ over all $\|x\|\leq\Rad\sqrt d$ is given by $\mathcal O(d\sqrt d/\sqrt n)$. Thus the bound~\eqref{intro:lip2}  in terms of the \emph{gradient} of $\phi$ gives a better dimension dependence than a bound in terms of \emph{the function $\phi$ itself}.\\

Thus to summarize, there are two key proof ingredients. The first ingredient is the Taylor expansion of $I(t)$, which gives us a simple explicit formula for the terms and remainder of the Laplace expansion for any $d$, in terms of Gaussian expectations. The second key ingredient~\eqref{intro:lip1},~\eqref{intro:lip2}, which uses the powerful theory of Gaussian concentration in high dimensions, gives us tight control on the terms and remainder, in powers of $d^2/n$. 

The bound~\eqref{intro:lip2}, as well as~\eqref{intro:lip1} with $k=3$, were a key feature in our earlier work~\cite{katskew} proving $d^2\ll n$ is generically necessary and sufficient for accuracy of the Laplace approximation to probability densities of the form $\pi\propto e^{-nv}$. Thus to some extent, the present work builds on these earlier results.

\paragraph{Proof techniques from prior works}
We first discuss past approaches that have been used to derive the coefficients of the Laplace asymptotic expansion in the univariate and multivariate (but fixed $d$) settings. 
The proof strategy closest to ours can be found in~\cite[Chapter 4.4]{de1981asymptotic}, for the one-dimensional Laplace expansion. The author expands the series in~\eqref{F-r-exp} in powers of $n^{-1/2}$ to obtain coefficients given by Gaussian integrals, but does not provide explicit formulas. This is also essentially the approach taken by~\cite{raudenbush2000maximum} in the context of a multivariate Laplace expansion for a particular statistical model. 

Another method which allows one to write the coefficients as Gaussian expectations is to apply the Morse Lemma, which states that there is a change of variables $x\mapsto y$ such that $v(x)=\frac12\sum_{j=1}^dy_j^2$~\cite[Chapter 9]{wongbook}. However, this change of variables is somewhat opaque, and derivatives of this coordinate transformation are required to compute the coefficients. Obtaining explicit expressions for these derivatives is very involved in high dimensions. 

The most prevalent proof of the Laplace expansion in one dimension relies on a different change of variables, which brings the integral into the form $\int f(y)e^{-n y}dy$; see~\cite{wojdylo2006computing},~\cite[pp. 36-39]{erdelyi1956asymptotic}, and~\cite[pp. 80-82]{olver1997asymptotics}. In fact this proof technique has also been extended to the multivariate case~\cite{kirwin2010higher,fulks1961asymptotics}. In this approach, only an asymptotic expansion of the functions $f$ and $v$ is required, rather than smoothness. Explicit expressions for the coefficients have been obtained using this method by~\cite{kirwin2010higher}, for the multivariate expansion. 

The aforementioned works have all considered the fixed $d$ regime. In the high-dimensional regime, the main challenge is to obtain tight, dimension-dependent upper bounds on the terms and remainder. Most prior works in the high-dimensional regime (discussed in Section~\ref{intro:stateofart}) have obtained explicit bounds only on the zeroth order Laplace expansion. These works have shown that $|\Rem_1|$ can be bounded in terms of $d^3/n$ (at best), and the proof techniques seem to involve the coarser bounds on the lefthand sides of~\eqref{intro:lip1} and~\eqref{intro:lip2} mentioned above; see for example~\cite{barber2016laplace} and~\cite{lapinski2019multivariate}. 

\subsection{Step 1: Taylor expansion of local integral}\label{sec:taylor} Below, we will Taylor expand the function $I(t)$ from~\eqref{I-def} about $t=0$. To that end, we first prove some properties about the derivatives of $\rr$ and $F$.
\begin{lemma}\label{lma:W:Ck}
Let $\rr$ be as in~\eqref{r-def} and recall from~\eqref{fv-reg} that $v\in C^{2L+2}\l(\bar\U/\sqrt n\r)$. Then we have the following alternative representation for $\rr$:
\beq\label{rr-alt} \rr(t,x) = \frac{t}{2}\int_0^1\la\nabla^{3}v(qtx), x^{\otimes 3}\ra (1-q)^{2}dq \qquad\forall |t|\in [0,1/\sqrt n), \; x\in\U\eeq Also, for each $x\in\U$ the function $\rr(\cdot,x)$ is $2L$ times continuously differentiable in $(0,1/\sqrt n)$, and for $1\leq k\leq 2L$ we have 
\beq\label{partial-t}
\partial_t^k\rr(t,x)= \int_0^1\la\nabla^{k+2}v(qtx), x^{\otimes k+2}\ra q^k(1-q)dq\qquad\forall t\in(0,1/\sqrt n), \; x\in\U.
\eeq 
\end{lemma} See Appendix~\ref{app:sec:proof} for the proof. Using~\eqref{rr-alt},~\eqref{partial-t}, and that $v\in C^{2L+2}\l(\bar\U/\sqrt n\r)$, we conclude $\partial_t^k\rr$, $k=0,\dots,2L$ is uniformly continuous in $(0,1/\sqrt n)\times\U$. In turn, this implies $\partial_t^k\rr$ has a unique continuous extension to $t=0$, which we can obtain by simply setting $t=0$ in~\eqref{partial-t}. We compute explicitly that
\beq\label{partial-0}
\partial_t^k\rr(0,x)=\frac{\la\nabla^{k+2}v(0), x^{\otimes k+2}\ra}{(k+1)(k+2)},\qquad\forall x\in\U,\;\forall 1\leq k\leq 2L.
\eeq
Next we discuss $F$. From the fact that $f\in C^{2L}(\bar\U/\sqrt n)$, we immediately conclude that $F(\cdot, x)$ is $2L$ times differentiable in $(0,1/\sqrt n)$ for each $x\in\U$, with
\beq\label{partial-f}
\partial_t^kF(t,x) = \lla\nabla^kf(tx),\;x^{\otimes k}\rra.
\eeq
Furthermore, $\partial_t^kF$ is uniformly continuous in $(0,1/\sqrt n)\times\U$ for all $k=0,\dots, 2L$. This again follows from the fact that $f\in C^{2L}(\bar\U/\sqrt n)$.
\begin{corollary}\label{corr:Ik}The function $I$ from~\eqref{I-def} is $2L$ times differentiable in $(0,1/\sqrt n)$, and the functions $I^{(k)}$, $k=0,\dots,2L$ extend continuously to the closed interval $[0,1/\sqrt n]$.  Furthermore, \beq\label{Ikt}I^{(k)}(t) = \int_\U \partial_t^k\l(F(t, x)\,e^{-\rr(t, x)}\r)\gamma(dx)\eeq for all $k=0,1,\dots,2L$, $t\in(0,1/\sqrt n)$.
\end{corollary}
\begin{proof}We know that $F(\cdot, x)$ and $\rr(\cdot,x)$ are $2L$ times differentiable in $(0,1/\sqrt n)$ for each $x\in\U$ and that $\partial_t^kF,\partial_t^k\rr$ are uniformly continuous in $(0,1/\sqrt n)\times\U$ for all $k=0,\dots,2L$. Hence $\partial_t^kF,\partial_t^k\rr$ extend uniquely to continuous functions on $[0,1/\sqrt n]\times\bar\U$. This implies
$$\sup_{t\in(0,1/\sqrt n),x\in\U}|\partial_t^kF(t,x)| <\infty,\quad \sup_{t\in(0,1/\sqrt n),x\in\U}|\partial_t^k\rr(t,x)| <\infty,\quad\forall 0\leq k\leq 2L.$$ Together, these regularity conditions suffice to show the derivatives $I^{(k)}$ exist and are given as in~\eqref{Ikt} for all $k=0,1,\dots,2L$, $t\in(0,1/\sqrt n)$. In other words, we can exchange the integral and derivative. Furthermore, the above properties imply $\partial_t^k\l(F(t, x)\,e^{-\rr(t, x)}\r)$ is also uniformly continuous in $(0,1/\sqrt n)\times\U$ for all $k=0,\dots,2L$, which in turn implies $I^{(k)}$ is uniformly continuous in $(0,1/\sqrt n)$, for all $k=0,\dots,2L$. But then $I^{(k)}$ extends continuously to $[0,1/\sqrt n]$ for all $k=0,\dots,2L$.
\end{proof}

The following is the main result of this section.
\begin{lemma}\label{lma:Taylor} For $k=0,1,\dots,2L$ and $(t,x)\in[0,1/\sqrt n]\times\U$ define 
\begin{align}
\chi_k(t,x) &= \sum_{\ell=0}^k{k\choose\ell}\partial_t^{k-\ell}F(t,x)\,B_\ell\l(-\partial_t\rr(t,x), -\partial_t^2\rr(t,x),\dots, -\partial_t^\ell\rr(t,x)\r),\label{chi-def}
\end{align} where $\partial_t^k\rr$ and $\partial_t^kF$ are as in~\eqref{partial-t} and~\eqref{partial-f}, respectively.
Then there exists $s\in(0, 1/\sqrt n)$ such that $$I(1/\sqrt n)= f(0) + \sum_{k=1}^{2L-1}n^{-k/2}A_k + \kappa_L+\maus{L},$$ where
\begin{align}
A_k &= \frac{1}{k!}\E\l[\chi_k(0, Z)\r],\label{eq:Ak}\\
\kappa_L &= \frac{n^{-L}}{(2L)!}\int\chi_{2L}(s,\cdot)e^{-\rr(s,\cdot)}\ind_\U d\gamma,\label{kap-def}\\
\maus{L}   &= -\sum_{k=0}^{2L-1}\frac{n^{-k/2}}{k!}\int\chi_k(0,\cdot)\ind_{\U^c}d\gamma.\label{tauL-def}
\end{align}
\end{lemma}
\begin{remark}Note that the $A_k$ involve an integral of $\chi_k(0,x)$ over all $x\in\R^d$, while $\maus{L}$ involves an integral of $\chi_k(0,x)$ over $x\in\U^c$. Technically, $\chi_k(0,x)$ has only been defined for $x\in\U$. However, note that, using the formulas~\eqref{partial-0} and~\eqref{partial-f} for $\partial_t^\rr(0,x)$ and $\partial_t^kF(0,x)$, respectively, we get
\beq\label{chik0}
\chi_k(0,x) = \sum_{\ell=0}^k{k\choose\ell}\la\nabla^{k-\ell}f(0), x^{\otimes k-\ell}\ra\,B_\ell\l(\frac{-\la\nabla^{3}v(0), x^{\otimes 3}\ra}{2\cdot 3}, \dots,\frac{-\la\nabla^{\ell+2}v(0), x^{\otimes \ell+2}\ra}{(\ell+1)(\ell+2)}\r)
\eeq This is simply a polynomial in $x$ which is well-defined on $\R^d$, so $\chi_k(0,x)$ can be extended to $\R^d$.
\end{remark}

Using this lemma together with~\eqref{main-ratio} we arrive at the following decomposition:
\beqs\label{overall-decomp}
\frac{e^{n\u(\xmin)}\sqrt{\det H}}{(2\pi/n)^{d/2}}\int_{\R^d} g(x)e^{-n\u(x)}dx &= I(1/\sqrt n)+\maus{\U^c} \\
&= f(0) + \sum_{k=1}^{2L-1}n^{-k/2}A_k + \kappa_L+\maus{L} +\maus{\U^c}.
\eeqs 
Note that the expression $A_{2k}=\frac{1}{(2k)!}\E[\chi_{2k}(0,Z)]$, with $\chi_{2k}(0,\cdot)$ as in~\eqref{chik0}, is precisely the same as in~\eqref{A-expect} of Theorem~\ref{thm:expand}. Thus to finish the proof of Theorem~\ref{thm:expand}, it remains to show that $A_{2k+1}=0$ for all $k$, and to bound the coefficients $A_{2k}$ and remainder terms $\kappa_L,\maus{L},\maus{\U^c}$. We do the former in Section~\ref{sec:lead}, where we also derive the explicit formula for the $A_{2k}$. We bound the coefficients and remainder in Section~\ref{sec:rem}.
\begin{proof}
Corollary~\ref{corr:Ik} shows that Taylor's theorem can be applied to the function $I(t)$ defined in~\eqref{I-def}. Specifically, for some $s\in(0,n^{-1/2})$, we have
\beqs\label{Taylor-I} 
I(1/\sqrt n) =& \sum_{k=0}^{2L-1}\frac{n^{-k/2}}{k!}I^{(k)}(0) + \frac{n^{-L}}{(2L)!}I^{(2L)}(s)\\
=&\sum_{k=0}^{2L-1}\frac{n^{-k/2}}{k!}\int_\U\partial_t^k\l(F(t, x)\,e^{-\rr(t, x)}\r)\bigg\vert_{t=0}\gamma(dx)\\
& +  \frac{n^{-L}}{(2L)!}\int_\U\partial_t^{2L}\l(F(t, x)\,e^{-\rr(t, x)}\r)\bigg\vert_{t=s}\gamma(dx).\eeqs
To compute $\partial_t^k(Fe^{-\rr})$, we use the product rule and the following formula for the derivatives of the exponential of a function. Namely, if $h\in C^k(\R)$, then
\beq
\frac{d^k}{dt^k}e^{h(t)} = e^{h(t)}B_k\l(h'(t), h''(t),\dots, h^{(k)}(t)\r),
\eeq where $B_k$ is the $k$th complete Bell polynomial defined in~\eqref{def:bell}. See Chapter 2 in~\cite{combinatoricsbook}. This result follows from Faa di Bruno's formula on the derivatives of $f\circ h$, with $f(x)=e^x$. See~\cite{combinatoricsbook, comtet2012advanced} for more on this topic. 
Thus
\beqs\label{partial-tk}
\partial_t^ke^{-\rr(t,x)} &= e^{-\rr(t,x)}B_k\l(-\partial_t\rr(t,x), -\partial_t^2\rr(t,x),\dots, -\partial_t^k\rr(t,x)\r)\eeqs for all $x\in\U$ and $k=1,\dots,2L$. We now have
\beq\label{partial-Fr}
\partial_t^k\l(F(t,x)e^{-\rr(t,x)}\r) = \sum_{\ell=0}^k{k\choose\ell}\partial_t^{k-\ell}F(t,x)\partial_t^\ell e^{-\rr(t,x)} = \chi_k(t,x)e^{-\rr(t,x)},
\eeq where the first equality is by the product rule and in the second equality, we substituted~\eqref{partial-tk}. Using~\eqref{partial-Fr} in~\eqref{Taylor-I} and recalling from~\eqref{r-def} that $\rr(0,\cdot)\equiv0$ finally gives 
\beqs\label{Taylor-III} I(1/\sqrt n) = \sum_{k=0}^{2L-1}\frac{n^{-k/2}}{k!}\int_\U\chi_k(0,x)\gamma(dx)+  \frac{n^{-L}}{(2L)!}\int_\U\chi_{2L}(s,x)e^{-\rr(s,x)}\gamma(dx).\eeqs
To conclude, we write $\int_{\U}=\int_{\R^d} - \int_{\U^c}$ for the integrals in the sum over $k$, and recognize that $\int_{\R^d}\chi_k(0,\cdot)d\gamma=\E[\chi_k(0,Z)]$. Note in particular that $\chi_0(t,x)=F(t,x)$, so that $\chi_0(0,x)=f(0)$.\end{proof}

\subsection{Step 2: bounds on terms and remainder}\label{sec:rem}
In this section, we prove the bounds on $A_{2k},\kappa_L,\maus{L}$. The bound on $\maus{\U^c}$ follows from~\eqref{v-av-growth-ii} and a gamma integral calculation; see Lemma~\ref{aux:gamma}. 
First we have the following preliminary bounds on $A_{2k}$, $\kappa_L$, and $\maus{L}$, which follow from the definitions~\eqref{eq:Ak}, \eqref{kap-def}, \eqref{tauL-def}, and from Cauchy-Schwarz:
\begingroup
\addtolength{\jot}{0.5em}
\beqs\label{alltogeth}
n^{-k}|A_{2k}| &\leq n^{-k}\|\chi_{2k}(0,\cdot)\|_1,\\
|\kappa_L| &\leq \sup_{t\in[0,1/\sqrt n]}n^{-L}\l \|\chi_{2L}(t,\cdot)\ind_\U\r \|_2 \,\l \|e^{-\rr(t,\cdot)}\ind_\U\r\|_2,\\
|\maus{L}| &\les_L \max_{0\leq k\leq 2L-1}n^{-k/2}\l \|\chi_k(0,\cdot)\r \|_2\sqrt{\gamma(\U^c)}.
\eeqs
\endgroup
Recall that $\gamma$ is the standard Gaussian distribution, and $\U=\{\|x\|\geq\Rad\sqrt d\}$. Here and below, we use the notation 
$$\|f\|_q:=\l(\int |f|^qd\gamma\r)^{1/q}.$$
For $\gamma(\U^c)$ we use the standard Gaussian tail bound $\gamma(\U^c) \leq \exp\l(-(\Rad-1)^2d/2\r)$, which follows e.g. from Example 2.12 of~\cite{boucheron2013concentration} (recalling that $\Rad>1$). There are essentially only two other distinct quantities arising in the righthand sides of~\eqref{alltogeth}: 
$$n^{-k/2}\|\chi_k(t,\cdot)\ind_\Omega\|_2,\qquad \|e^{-\rr(t,\cdot)}\ind_\U\|_2,$$ where $\Omega$ is either $\U$ or $\R^d$. These quantities are bounded in Propositions~\ref{prop:exp} and~\ref{corr:B-Lp} in the following two subsections, respectively. At the heart of these two propositions are the two key bounds discussed in the proof outline in Section~\ref{sec:outline}. The former proposition uses~\eqref{intro:lip2} and the latter uses~\eqref{intro:lip1}.

Combining the bounds from these propositions with the Gaussian tail bound and~\eqref{alltogeth} finishes the proof of the bounds~\eqref{leading-bound},~\eqref{rem-kap},~\eqref{rem-tauL} on $A_{2k},\kappa_L,\maus{L}$, respectively.
\subsubsection{Exponential bound}
\begin{prop}\label{prop:exp} It holds
\beq \label{eq:exp}
\sup_{t\in(0,1/\sqrt n)}\|e^{-\rr(t,\cdot)}\ind_\U \|_{2}\leq \exp\l(\l[\Rad^4c_3(\Rad)^2+c_4(\Rad)\r]\epsilon^2\r).
\eeq
 \end{prop}
\begin{proof}
We use the key estimate~\eqref{corr:exp:gauss} of Corollary~\ref{corr:gauss}, which was presented in simplified form as the key bound~\eqref{intro:lip2} in the proof outline:
\beq\label{keyexpapply}
\|e^{-\rr(t,\cdot)}\ind_\U \|_{2} \leq\exp\l(-\gamma(\U)^{-1}\int_\U \rr(t,\cdot)d\gamma + \sup_{x\in\U}\|\nabla_x\rr(t,x)\|^2\r)
\eeq
Now, for $t\neq0$ we have $\rr(t,x) = v(tx)/t^2 - \|x\|^2/2$ and therefore $\nabla_x\rr(t,x)=\nabla v(tx)/t - x$. Taylor expanding $\rr$ and $\nabla_x\rr$ about $x=0$ for fixed $t>0$ gives
\begin{align}\label{texp-r}
\rr(t,x)& =\frac{t}{3!}\la\nabla^3v(0), x^{\otimes3}\ra + \frac{t^2}{4!}\la\nabla^4v(\xi tx), x^{\otimes4}\ra,\\
\nabla_x \rr(t,x) &= \frac t2\la\nabla^3v(\xi tx), x^{\otimes2}\ra,\label{texp-nablar}\end{align} for some $\xi\in[0,1]$. From~\eqref{texp-r} we get that 
\beqsn
\gamma(\U)^{-1}\l|\int\rr(t,\cdot)d\gamma\r| &\leq \gamma(\U)^{-1}\frac{t^2}{4!}\int_\U \l|\la\nabla^4v(\xi tx), x^{\otimes4}\ra\r| \gamma(dx) \\
&\leq \frac{1}{24n}c_4(\Rad)\gamma(\U)^{-1}\E[\|Z\|^4] \leq c_4(\Rad)\frac{d^2}{n}\eeqsn for all $t\in(0,1/\sqrt n)$. Here we used that $\int_\U\la\nabla^3v(0), x^{\otimes3}\ra \gamma(dx)=0$ and that $\gamma(\U)\geq1/2$ by the Gaussian tail bound discussed above (since $\Rad\geq4$). From~\eqref{texp-nablar} we get that
$$\sup_{x\in\U}\|\nabla_x\rr(t,x)\|^2 \leq \l(\frac{1}{\sqrt n}c_3(\Rad)(\Rad \sqrt d)^2\r)^2 = c_3(\Rad)^2\frac{\Rad^4d^2}{n}$$  for all $t\in(0,1/\sqrt n)$. Combining the above two estimates in~\eqref{keyexpapply} concludes the proof.
\end{proof}

\subsubsection{Bound on $\chi_k$}In this section, we prove the following main proposition.
\begin{prop}\label{corr:B-Lp}
It holds
\begin{align}
\label{Bk0-bd}
n^{-k/2}\|\chi_k(0,\cdot)\|_2 &\les_{k} \A_k(0)\epsilon^k,\qquad\forall\;1\leq k\leq 2L,\\
\label{delta-k-Bk-bd}
\sup_{0\leq t\leq 1/\sqrt n}n^{-L}\|\chi_{2L}(t,\cdot)\ind_{\U}\|_2 &\les_{L}  \A_{2L}(\Rad)\epsilon^{2L}.
\end{align} 
\end{prop}
To prove this proposition we start with the following inequality, recalling the definition of $\chi_k$ from~\eqref{chi-def}:
\beqs\label{chik-prelim}
n^{-k/2}\|\chi_k\ind_\Omega\|_2 &\les_k \sum_{\ell=0}^kn^{-\ell/2}\l\|B_\ell\l(-\partial_t\rr, \dots, -\partial_t^\ell\rr\r)\ind_\Omega\r\|_4n^{-(k-\ell)/2}\|\partial_t^{k-\ell}F\ind_\Omega\|_4\\
&\les_k\sum_{\ell=0}^kB_\ell\l(n^{-1/2}\l \|\partial_t\rr\ind_\Omega\r \|_{4\ell^2}\,,\, \dots, \,n^{-\ell/2} \|\partial_t^\ell\rr\ind_\Omega \|_{4\ell^2}\r)n^{-(k-\ell)/2}\l \|\partial_t^{k-\ell}F\ind_\Omega\r \|_4.
\eeqs
Above, we have used the shorthand $\|u\ind_\Omega\|_q = \|u(t,\cdot)\ind_\Omega\|_q$. The first upper bound is by Cauchy-Schwarz, and the second upper bound is proved in Lemma~\ref{lma:Bell:Lp}, using the structure of the Bell polynomials. Thus it remains to bound the quantities $\|\partial_t^k\rr\ind_\Omega\|_q$ and $\|\partial_t^kF\ind_\Omega\|_q$ for some $q\geq2$. Specifically, we are interested in $\|\partial_t^k\rr(0,\cdot)\|_q$ and $\|\partial_t^k\rr(t,\cdot)\ind_\U\|_q$, and similarly for $F$.

At this point is useful to recall the formulas~\eqref{partial-0} and~\eqref{partial-t} for $\partial_t^k\rr(0,\cdot)$ and $\partial_t^k\rr(t,\cdot)$, respectively, the formula~\eqref{partial-f} for $\partial_t^kF(t,\cdot)$, and Definition~\ref{def:alph} of the $\alpha$'s.
\begin{lemma}\label{lma:partial-t-r-bounds}
For $\rr$, we have the following bounds for all $k=1,2,3,\dots$:
\begin{align}
\|\partial_t^k\rr(0,\cdot)\|_q &\les_{k,q} \alpha_k(0)d^k=
\begin{cases}
c_{k+2}(0){\sqrt d}^{k+2},&\text{$k$ even},\\
c_{k+2}(0){\sqrt d}^{k+1},&\text{$k$ odd},
\end{cases},\label{partial-r-0}\\
\sup_{t\in(0,1/\sqrt n)}\|\partial_t^k\rr(t,\cdot)\ind_\U\|_q &\les_{k,q} \alpha_k(\Rad)d^k=
\begin{cases}
c_{k+2}(\Rad){\sqrt d}^{k+2},&\text{$k$ even},\\
\bar c_{k+2}(\Rad){\sqrt d}^{k+1},&\text{$k$ odd.}
\end{cases}\label{partial-r-R}
\end{align} For $F$, we have the following bounds for all $k=0,1,2,\dots$
\begin{align}
\|\partial_t^kF(0,\cdot)\|_q &\les_{k,q} \alpha_{k,g}(0)d^k=
\begin{cases}
c_{k,g}(0){\sqrt d}^{k},&\text{$k$ even},\\
c_{k,g}(0){\sqrt d}^{k-1},&\text{$k$ odd},
\end{cases}\label{partial-F-0}\\
\sup_{t\in(0,1/\sqrt n)}\|\partial_t^kF(t,\cdot)\ind_\U\|_q &\les_{k,q}\alpha_{k,g}(\Rad)d^k=
\begin{cases}
c_{k,g}(\Rad){\sqrt d}^{k},&\text{$k$ even},\\
\bar c_{k,g}(\Rad){\sqrt d}^{k-1},&\text{$k$ odd.}
\end{cases}\label{partial-F-R}
\end{align}
\end{lemma}
\begin{remark}[Redefinition of the $\alpha$'s]\label{rk:alpha} In the below proof, we show that the bound for even $k$ in~\eqref{partial-r-R} also holds for odd $k$. Therefore, for odd $k$, we can get a tighter overall bound by taking the minimum of $c_{k+2}(\Rad){\sqrt d}^{k+2}$ and $\bar c_{k+2}(\Rad){\sqrt d}^{k+1}$. Namely, it holds
\beq\label{newalph}
\sup_{t\in(0,1/\sqrt n)}\|\partial_t^k\rr(t,\cdot)\ind_\U\|_q \les_{k,q} \min\left(c_{k+2}(\Rad)\sqrt d, \bar c_{k+2}(\Rad)\right){\sqrt d}^{k+1}
\eeq
for $k$ odd. Similarly, for the bounds on the derivatives of $F$, we actually have
\beq\label{newalphg}
\sup_{t\in(0,1/\sqrt n)}\|\partial_t^kF(t,\cdot)\ind_\U\|_q \les_{k,q} \min\left(c_{k,g}(\Rad)\sqrt d, \bar c_{k,g}(\Rad)\right){\sqrt d}^{k-1}
\eeq when $k$ is odd.
A consequence of this is that we can redefine $\alpha_k(\Rad)$ to be such that $\alpha_k(\Rad)d^k$ is equal to the righthand side of~\eqref{newalph}, matching the structure of the result in~\eqref{partial-r-R}. Similarly, we can redefine $\alpha_{k,g}(\Rad)$ to be such that $\alpha_{k,g}(\Rad)d^k$ is equal to the righthand side of~\eqref{newalphg}. Our main Theorem~\ref{thm:expand} will continue to be valid with this alternative definition of the $\alpha_k$ and $\alpha_{k,g}$ for $k$ odd. Explicitly, the new definition is as follows:
\begingroup
\addtolength{\jot}{1em}
\beqs
\alpha_{k}(r)&=d^{(1-k)/2}\min\l(\bar c_{k+2}(r), \sqrt dc_{k+2}(r)\r), \qquad k=3,5,7,\dots\\
\alpha_{k,g}(r) &= d^{-(1+k)/2}\min\l(\bar c_{k,g}(r), \sqrt dc_{k,g}(r)\r), \qquad k=1,3,5,\dots.
\eeqs \endgroup Note that $(1-k)/2 = 1-\lceil k/2\rceil$ and $-(1+k)/2 = -\lceil k/2\rceil$ for $k$ odd, so the power of $d$ downweighting is the same above as in Definition~\ref{def:alph}. The difference is that we have replaced $\bar c_{k+2}(r)$ by $\min(\bar c_{k+2}(r),\sqrt dc_{k+2}(r))$. For $r=0$, the new definition coincides with the old one since $\bar c_k(0)=c_k(0)$, and therefore the minimum of $\bar c_k(0)$ and $\sqrt dc_k(0)$ is $\bar c_k(0)=c_k(0)$.\end{remark}
\begin{proof}
Using~\eqref{partial-0} we see that for even $q$ the $\|\partial_t^k\rr(0,\cdot)\|_q$ are bounded as follows:
\beqsn
\|\partial_t^k\rr(0,\cdot)\|_q \leq \E\l[\lla\nabla^{k+2}v(0), Z^{\otimes k+2}\rra^q\r]^{\frac1q}.
\eeqsn We can now use the key result~\eqref{intro:lip1} (proved in Proposition~\ref{lma:pk-bds}, along with a bound in the case of even $k$) to conclude~\eqref{partial-r-0}. To bound $\|\partial_t^k\rr(t,\cdot)\ind_\U\|_q$, we proceed as follows. First, we have the following bound using~\eqref{partial-t}:
$$
|\partial_t^k\rr(t,x)|\les_k c_{k+2}(\Rad)\|x\|^{k+2},\qquad\forall x\in\U,\; t\in(0,1/\sqrt n).
$$ This holds for all $k$. For odd $k$, in order to take advantage of the refined bound on $\|\partial_t^k\rr(0,\cdot)\|_q$, we Taylor expand further in $t$:
\beqsn
|\partial_t^k\rr(t,x)| &\leq|\partial_t^k\rr(0,x)|+t|\partial_t^{k+1}\rr(\xi, x)|\\
&\leq|\partial_t^k\rr(0,x)|+n^{-1/2}c_{k+3}(\Rad)\|x\|^{k+3}\qquad\forall x\in\U,\,t\in(0,1/\sqrt n).\eeqsn Combining all of the above estimates, including~\eqref{partial-r-0}, we obtain
\beq
\|\partial_t^k\rr(t,\cdot)\ind_\U\|_q \les_{k,q}
\begin{cases}
c_{k+2}(\Rad){\sqrt d}^{k+2},&k=1,2,3,\dots\\
c_{k+2}(0){\sqrt d}^{k+1}+ n^{-1/2}c_{k+3}(\Rad){\sqrt d}^{k+3},&\text{$k$ odd.}
\end{cases}
\eeq 
It remains to note that, upon factoring out ${\sqrt d}^{k+1}$ in the second line for $k$ odd, we obtain $c_{k+2}(0)+(d/\sqrt n)c_{k+3}(\Rad) = \bar c_{k+2}(\Rad)$. This finishes the proof of~\eqref{partial-r-R}. The proof of~\eqref{partial-F-0} and~\eqref{partial-F-R} is exactly analogous, so we omit it.
\end{proof}
We can now finish the proof of Proposition~\ref{corr:B-Lp}.
\begin{proof}[Proof of Proposition~\ref{corr:B-Lp}]Multiplying the inequalities from the above lemma by $n^{-k/2}$ gives 
\beqs\label{alpha-k-bd}
n^{-k/2}\|\partial_t^k\rr(0,\cdot)\|_q &\les_{k,q} \alpha_k(0)\epsilon^k,\qquad n^{-k/2}\|\partial_t^k\rr(t,\cdot)\ind_\U\|_q \les_{k,q} \alpha_k(\Rad)\epsilon^k,\\
n^{-k/2}\|\partial_t^kF(0,\cdot)\|_q &\les_{k,q} \alpha_{k,g}(0)\epsilon^k,\qquad n^{-k/2}\|\partial_t^kF(t,\cdot)\ind_\U\|_q \les_{k,q} \alpha_{k,g}(\Rad)\epsilon^k.
\eeqs
We now combine~\eqref{alpha-k-bd} with the inequality~\eqref{chik-prelim} to get an overall bound on $\|\chi_k(t,\cdot)\ind_\Omega\|$. Taking $\Omega=\R^d$ and $t=0$ in~\eqref{chik-prelim} and applying the bounds in the first column of~\eqref{alpha-k-bd} gives
\beqsn
n^{-k/2}\|\chi_k(0,\cdot)\|_2 &\les_k\sum_{\ell=0}^kB_\ell\l(\alpha_1(0)\epsilon\,,\, \dots, \alpha_\ell(0)\epsilon^\ell\r)\alpha_{k-\ell,g}(0)\epsilon^{k-\ell}\\
&\les_k\sum_{\ell=0}^kB_\ell\l(\alpha_1(0), \dots, \alpha_\ell(0)\r)\alpha_{k-\ell,g}(0)\epsilon^{k} = \A_k(0)\epsilon^k,\eeqsn as desired. To get the second line we used Lemma~\ref{lma:Bellprops}. This proves~\eqref{Bk0-bd}, and~\eqref{delta-k-Bk-bd} is proved analogously.
\end{proof}
\subsection{Formula for the coefficients}\label{sec:lead}Combining the definition of $\chi_k(0,x)$ and $A_k$ from~\eqref{chik0} and~\eqref{eq:Ak}, respectively, we obtain
\beqs\label{def:Ak}
A_k= \frac{1}{k!}\sum_{\ell=0}^k{k\choose\ell}\E\l[\la\nabla^{k-\ell}f(0), Z^{\otimes k-\ell}\ra\,B_\ell\l(\frac{-\lla\nabla^{3}v(0), Z^{\otimes 3}\rra}{2\cdot 3}, \dots,\frac{-\lla\nabla^{\ell+2}v(0), Z^{\otimes \ell+2}\rra}{(\ell+1)(\ell+2)}\r)\r].\eeqs In this section, we derive an explicit formula for the $A_k$ and show that $A_{k}=0$ for odd $k$. To do so, we first state a useful alternative representation of the Bell polynomials; see Appendix~\ref{app:sec:proof} for the proof.
\begin{lemma}\label{lma:alt:bell}For all $\ell\geq1$, the complete Bell polynomial $B_\ell$ defined in~\eqref{def:bell} has the following equivalent form:
\beq\label{eq:Bk:equiv}
B_\ell(x_1,\dots,x_\ell)=\ell!\sum_{r=1}^\ell\frac{1}{r!}\;\sum_{\substack{m_1+m_2+\dots+m_r=\ell\\ m_1,\dots,m_r\geq1}}\,\prod_{j=1}^r\frac{x_{m_j}}{m_j!}
\eeq
\end{lemma} 
This gives, for all $\ell\geq1$, that
\beqs\label{Bk-official}
B_\ell&\l(\frac{-\la\nabla^{3}v(0), x^{\otimes 3}\ra}{2\cdot 3}, \dots,\frac{-\la\nabla^{\ell+2}v(0), x^{\otimes \ell+2}\ra}{(\ell+1)(\ell+2)}\r)\\
&=\ell!\sum_{r=1}^\ell\frac{(-1)^r}{r!}\;\sum_{\substack{m_1+m_2+\dots+m_r=\ell\\ m_1,\dots,m_r\geq1}}\;\prod_{j=1}^r\frac{\la\nabla^{m_j+2}v(0), x^{\otimes m_j+2}\ra}{(m_j+2)!}
\eeqs 
Therefore, separating out the case $\ell=0$ in~\eqref{def:Ak} from the rest, we obtain
\beqs\label{Akr0}
A_k = &\frac{1}{k!}\E\l[\la\nabla^{k}f(0), Z^{\otimes k}\ra\r] \\
&+\sum_{\ell=1}^k\sum_{r=1}^\ell\frac{(-1)^r}{r!}\;\sum_{\substack{m_1+m_2+\dots+m_r=\ell\\ m_1,\dots,m_r\geq1}}\E\l[\frac{\la\nabla^{k-\ell}f(0), Z^{\otimes k-\ell}\ra}{(k-\ell)!}\prod_{j=1}^r\frac{\la\nabla^{m_j+2}v(0), Z^{\otimes m_j+2}\ra}{(m_j+2)!}\r]
\eeqs 
We now compute these expectations explicitly to prove the following theorem. To state the theorem, we introduce some multi-index notation. For a multi-index $\balph=(\alpha^1,\dots,\alpha^d)$, with $\alpha^j\geq0$ for all $j$, we define $|\balph|=\alpha^1+\dots+\alpha^d$ and $\balph!=\alpha^1!\dots\alpha^d!$. Similarly, $\balph!! =\alpha_1!!\dots\alpha_d!!$. If $\balph$ and $\bbet$ are two multi-indices, then $\balph+\bbet$ is their entrywise sum, and similarly for differences. We let $\mathbf 1= (1,\dots, 1)$ be the multi-index of all ones. If $x\in\R^d$ then $x^\balph=x_1^{\alpha^1}\dots x_d^{\alpha^d}$. We let $\partial^\balph f(x)=\partial_{x_1}^{\alpha^1}\dots\partial_{x_d}^{\alpha^d}f(x)$ whenever this partial derivative exists.

\begin{theorem}\label{lma:leading}For all $k\geq1$, it holds
\beqs\label{Akf-thm}
A_k = \sum_{|\bbet|=k}&\frac{(\bbet-\mathbf{1})!!}{\bbet!}\partial^\bbet f(0)\even(\bbet)\\
+\sum_{\ell=1}^k&\sum_{r=1}^\ell\frac{(-1)^r}{r!}\;\sum_{\substack{m_1+m_2+\dots+m_r=\ell\\ m_1,\dots,m_r\geq1}}\;\sum_{|\bbet|=k-\ell,|\balph_1|=m_1+2,\dots,|\balph_r|=m_r+2}\\
&\even(\bbet+\balph_1+\dots + \balph_r)\frac{(\bbet+\balph_1+\dots+\balph_r-\mathbf{1})!!}{\bbet!\balph_1!\dots\balph_r!}\partial^\bbet f(0)\partial^{\balph_1} v(0)\dots\partial^{\balph_r}v(0),
\eeqs where $0!! =0!=1$ and $\even(\balph)=1$ if $\alpha_i$ is even for each $i=1,\dots, d$, and $\even(\balph)=0$ otherwise.
\end{theorem} See Appendix~\ref{app:sec:proof} for the proof. These $A_{k}$ coincide with the coefficients obtained in the finite dimensional Laplace asymptotic expansion given in~\cite{kirwin2010higher}; see the formula in Theorem 1.1 of the arXiv version of~\cite{kirwin2010higher}, and Theorem 1.2 of the published version of this work. We have intentionally used the same notation. The only difference is that we have factored out the $(2\pi)^{d/2}/\sqrt{\det H}$, and that we have separated out the $\ell=0$ case into the first line of~\eqref{Akf-thm}. 

Note that if $|\balph|$ is odd, then all the indices $\alpha_i$ cannot be even, and hence $\even(\balph)=0$. But note that $|\bbet+\balph_1+\dots+\balph_r| = (k-\ell) +(m_1+2)+\dots+(m_r+2) = k-\ell + \ell + 2r = k+2r$ is odd if $k$ is odd. Therefore the sum in the second and third line is zero when $k$ is odd, and by the same reasoning, the sum in the first line is also zero when $k$ is odd. We conclude that $A_k=0$ for all odd $k$.
\begin{corollary}\label{corr:leading}
When $f=1$, the $A_k$ are given as follows for all $k\geq1$:
\beqs
A_k =  &\sum_{r=1}^k\frac{(-1)^r}{r!}\;\sum_{\substack{m_1+m_2+\dots+m_r=k\\ m_1,\dots,m_r\geq1}}\;\sum_{|\balph_1|=m_1+2,\dots,|\balph_r|=m_r+2}\\
&\even(\balph_1+\dots + \balph_r)\frac{(\balph_1+\dots+\balph_r-\mathbf{1})!!}{\balph_1!\dots\balph_r!}\partial^{\balph_1} v(0)\dots\partial^{\balph_r}v(0),
\eeqs 
\end{corollary}
\appendix
\section{The term $A_2n^{-1}$ and comparison to~\cite{shun1995laplace}}\label{app:shun}
We first give the formula for $A_2$ both as a Gaussian expectation and an explicit sum. Using~\eqref{A-expect} with $k=1$, and the fact that $B_0=1$, $B_1(s_1)=s_1$, and $B_2(s_1,s_2)=s_1^2+s_2$, we obtain
\beqs\label{A2f-explicit}
A_{2} = &\frac12\E\l[\la\nabla^2f(0), Z^{\otimes2}\ra\r] -\frac16\E\l[\la\nabla f(0), Z\ra\la\nabla^3v(0), Z^{\otimes3}\ra\r] \\
&+ \frac{f(0)}{72}\E\l[\la\nabla^3v(0), Z^{\otimes3}\ra^2\r]-\frac{f(0)}{24}\E\l[\la\nabla^4v(0), Z^{\otimes4}\ra\r]\\
 =&\frac12\Delta f(0)-\frac12\sum_{i,j=1}^d\partial_if(0)\partial_{ijj}^3v(0) + \frac{f(0)}{12}\|\nabla^3v(0)\|^2_{F} \\
&+ \frac{f(0)}8\sum_{i=1}^d\l(\partial_i\Delta v(0)\r)^2 - \frac{f(0)}8\sum_{i,j=1}^d\partial_i^2\partial_j^2v(0).
\eeqs Here, recall that $v(x)=\u(\xmin+H^{-1/2}x)-\u(\xmin)$ and $f(x)=g(\xmin+H^{-1/2}x)$.\\

The work~\cite{shun1995laplace} considers the magnitude of $A_2n^{-1}$ when $g\equiv1$ and therefore also $f\equiv1$. Let us write the explicit formula for $A_2$ in this case. Substituting $f\equiv1$ into~\eqref{A2f-explicit} gives
\beq\label{A2I}A_2 = \frac{1}{12}\|\nabla^3v(0)\|^2_{F} + \frac{1}8\sum_{i=1}^d\l(\partial_i\Delta v(0)\r)^2 - \frac{1}8\sum_{i,j=1}^d\partial_i^2\partial_j^2v(0).\eeq
We claim that $A_2n^{-1}$ coincides with the ``correction term" $\epsilon_0$ from (2) in~\cite{shun1995laplace}. To see that the formulas are the same, it is simplest to consider the case when $H=\nabla^2\u(\xmin)=I_d$, in which case $v=\u$, up to translation. We now use $\hat g^{ij}=\delta_{ij}$, $\hat g_{ijk}=\partial^3_{ijk}v(0)$, and $\hat g_{ijk\ell}=\partial^4_{ijk\ell}v(0)$ in (2) of~\cite{shun1995laplace}, to get precisely $A_2n^{-1}$.

In their Section 6 on GLMs (for which the authors use the term ``exponential models"), Shun and McCullagh consider essentially the same setting as we do in Section~\ref{sec:log}. (To follow their reasoning in Section 6, it is important to note that $g$ now denotes what was called $ng$ in Sections 1 and 2, and there is no more hat on the $g$). After a linear transformation by $H^{-1/2}$, they obtain $ g_{ij}=n\partial^2_{ij}v(0)=n\delta_{ij}$ and the higher order derivatives $g_{ijk}=n\partial^3_{ijk}v(0)$, $g_{ijk\ell}=n\partial^4_{ijk\ell}v(0)$ are precisely as in~\eqref{nablav-log}. See display (10) in~\cite{shun1995laplace} as well as the display and discussion both above and below (10).

The authors argue that $\epsilon_0=A_2n^{-1}$ scales as $d^3/n$ based on the fact that $\epsilon_0$ is given by a sum of $d^3$ terms of order $n^{-1}$.  Indeed, note that in~\eqref{A2I} the first two sums are both over $d^3$ terms (expanding the Laplacian in the second sum), and the third sum is over $d^2$ terms. However, simply counting the number of terms in a sum can yield too coarse of an estimate. Indeed, the bound~\eqref{RemL-log} from Corollary~\ref{corr:expand:log} implies $|A_2n^{-1}|=|\Rem_1 - \Rem_2|\les d^2/n$, disproving the claim that $A_2n^{-1}\sim d^3/n$. This is true under the assumptions of i.i.d. Gaussian features $X_i$, which is a special case of the assumption made in~\cite{shun1995laplace}.

We now give a direct proof that $A_2$ is bounded by $d^2$, based on the representation~\eqref{A2I} as an explicit sum. We claim that
$$|A_2| \leq d^2\l(\|\nabla^3v(0)\|^2 + \|\nabla^4v(0)\|\r),$$ and therefore, $|A_2|$ scales as $d^2$, since we know $\|\nabla^3v(0)\|, \|\nabla^4v(0)\|=\mathcal O(1)$ by Corollary~\ref{corr:c:log}. (Recall that $\|\nabla^kv(0)\|=\|\nabla^k\u(\xmin)\|_H=c_k(0)$.)

The third sum in~\eqref{A2I} is clearly seen to be bounded by $d^2\|\nabla^4v(0)\|$, since $|\partial^4_{ijk\ell}v(0)|\leq\|\nabla^4v(0)\|$ for all $i,j,k,\ell$. To bound the first term, let $T=\nabla^3v(0)$ and $T_i$ be the matrix whose $jk$th entry is $T_{ijk}$. Note that 
$$\|T_i\| = \sup_{\|u\|=\|v\|=1}\la T_i, u\otimes v\ra = \sup_{\|u\|=\|v\|=1}\la T, u\otimes v\otimes e_i\ra \leq \|T\|$$ for all $i$. Therefore, 
\beq\label{TF}
\|\nabla^3v(0)\|_F^2=\|T\|_F^2=\sum_{i=1}^d\mathrm{Tr}(T_i^\T T_i)\leq\sum_{i=1}^dd\|T_i^\T T_i\| \leq d^2\max_{1\leq i\leq d}\|T_i\|^2 \leq d^2\|T\|^2.\eeq Finally, note that $\sum_{i=1}^d\l(\partial_i\Delta v(0)\r)^2 = \|\la T, I_d\ra\|^2,$ where $\la T,I_d\ra\in\R^d$ is the vector with entries $\la T, I_d\ra_i = \sum_{j=1}^dT_{ijj}$. Now, we have
\beqsn
\|\la T,  I_d\ra\| &= \sup_{\|u\|=1}\la T, I_d\ra^\T u = \sup_{\|u\|=1}\la T, I_d\otimes u\ra\\
&\leq\sup_{\|u\|=1}\sum_{j=1}^d|\la T, e_j\otimes e_j\otimes u\ra| \leq d\|T\|\eeqsn and hence $$\sum_{i=1}^d\l(\partial_i\Delta v(0)\r)^2= \|\la T,  I_d\ra\| ^2 \leq d^2\|T\|^2.$$
 \section{Proofs from Section~\ref{sec:main}}\label{app:main}
\begin{proof}[Proof of Lemma~\ref{lma:powers}]
The Bell polynomials have the following property: $x^jB_j(y_1,y_2,\dots,y_j)=B_j(x^1y_1,x^2y_2,\dots,x^jy_j)$; see Lemma~\ref{lma:Bellprops}. Therefore,
\beq\label{Aktauk}\A_j\tau^{-j} = \sum_{\ell=0}^{j}\alpha_{j-\ell,g}\tau^{-(j-\ell)}B_{\ell}(\alpha_1\tau^{-1},\dots,\alpha_\ell\tau^{-\ell}),\eeq where we have omitted the radius argument from $\A_j$ and the $\alpha_j$'s for brevity. 


To bound $\A_{j}(\Rad)$, first note that~\eqref{ck-ev-R} and~\eqref{ckg-ev-R} and Definition~\ref{def:alph} of the $\alpha$'s imply $\alpha_k(\Rad)\tau^{-k}\les_{L} 1$ and $\alpha_{k,g}(\Rad)\tau^{-k}\les_{L} 1$ for even $k$. For odd $k$, we have
\beqsn
\alpha_{k,g}(\Rad)\tau^{-k} &= (c_{k,g}(0)+\epsilon c_{k+1,g}(\Rad))\tau^{- k}/{\sqrt d}^{k+1}\\
&=c_{k,g}(0)\tau^{- k}/{\sqrt d}^{k+1} + (\epsilon\tau)c_{k+1,g}(\Rad)\tau^{-(k+1)}/{\sqrt d}^{k+1}\les_{L} 1
\eeqsn using~\eqref{ckg-0} and~\eqref{ckg-ev-R}. A similar argument gives $\alpha_k(\Rad)\tau^{-k}\les_{L}1$ for odd $k$. We conclude that $\alpha_k(\Rad)\tau^{-k}\les_{L}1$ and $\alpha_{k,g}(\Rad)\tau^{-k}\les_{L}1$ for all $k\leq j$. It follows from~\eqref{Aktauk} that $\A_{j}(\Rad)\tau^{-j}\les_{L} 1$, as desired.

To prove the second bound of~\eqref{Ak-rad}, we apply~\eqref{ck-0}, $k=3$ and~\eqref{ck-ev-R}, $k=4$ to get $c_3(0)\les\tau$ and $c_4(\Rad)\les \tau^{2}$.
Next, a Taylor expansion gives 
$$c_3(\Rad)\leq c_3(0)+\Rad\sqrt{d/n}c_4(\Rad)\leq \tau + \Rad\tau^2\epsilon.$$ Therefore
\beqsn
\Rad^2c_3(\Rad)\epsilon &\leq \Rad^2\tau\epsilon + (\Rad^2\tau\epsilon)^2 \les (\Rad^2\tau\epsilon\vee1)^2,\\
c_4(\Rad)\epsilon^2 &\leq (\tau\epsilon)^2\leq 1.
\eeqsn Finally, note from~\eqref{tau-ep} and~\eqref{Rad-simpler} that $\Rad^2\tau\epsilon\vee1\les_L1$. Combining these estimates gives the desired result.
\end{proof} 
\paragraph{Derivation of formula for $A_{2k}$ in~\eqref{quart-decomp}.} First note that $\u$ has a global minimizer at $\xmin=0$, with $H=\nabla^2\u(0)=I_d$. Next, note that when $f\equiv1$, the formula~\eqref{A-expect} for $A_{2k}$ reduces to
\beq\label{A2k-part}A_{2k}=\frac{1}{(2k)!}\E\l[ B_{2k}\l(\frac{-\lla\nabla^{3}v(0), Z^{\otimes 3}\rra}{2\cdot 3}, \dots,\frac{-\lla\nabla^{2k+2}v(0), Z^{\otimes 2k+2}\rra}{(2k+1)(2k+2)}\r)\r].\eeq Using~\eqref{A2k-part}, the fact that $\u=v$ (since $H=I_d$), and~\eqref{lanabla4}, we obtain
\beqsn
A_{2k}&=\frac{1}{(2k)!}\E\l[B_{2k}\l(\-\lla\nabla^{3}v(0), Z^{\otimes 3}\rra/(2\cdot 3), \dots,-\lla\nabla^{2k+2}v(0), Z^{\otimes 2k+2}\rra/((2k+1)(2k+2))\r)\r]\\
&=\frac{1}{(2k)!}\E\l[B_{2k}\l(0, -\frac{1}{12}\la\nabla^4v(0), Z^{\otimes4}\ra, 0, \dots,0\r)\r] = \frac{1}{(2k)!}\E\l[B_{2k}\l(0, -\frac{\|Z\|^4}{12},0,\dots,0\r)\r].\eeqsn But now, using the formula~\eqref{def:bell} for the Bell polynomials, we have for all $k\geq1$ that
\beqsn
A_{2k}&=\frac{1}{(2k)!}\E\l[B_{2k}(0,-\|Z\|^4/12, 0,\dots,0)\r] =\frac{\E\l[(-\|Z\|^4/12)^{k}\r]}{(2!)^kk!}\\
&=\frac{(-1)^k}{k!(24)^k}\E\l[\|Z\|^{4k}\r] = \frac{(-1)^k}{k!(24)^k}\frac{2^{2k}\Gamma(2k+d/2)}{\Gamma(d/2)} = \l(\frac{-1}{6}\r)^k\frac{\Gamma(2k+d/2)}{k!\Gamma(d/2)}.\eeqsn To compute the Gaussian expectation, we used that $\E[\|Z\|^{4k}]$ is the $(2k)$th moment of $\chi^2_d$, a chi-squared distribution with $d$ degrees of freedom. From here, the definition of the gamma function immediately gives that $C_kd^{2k}\leq |A_{2k}| \leq C_k'd^{2k}$ for some $C_k,C_k'>0$.

\begin{proof}[Proof of Proposition~\ref{prop:quart}]
Recall the minimizer is $\xmin=0$, with $H=I_d$. We have $\nabla^k\u(x)\equiv0$ for all $k>4$, while $\nabla^3\u(0)=0$ and $\nabla^4\u(x)=\nabla^4\u(0)$ for all $x$, implying $c_4(r)=c_4(0)$ for all $r\geq0$.  
Furthermore, we have
\beq\label{lanabla4}\la\nabla^4\u(0), x^{\otimes4}\ra = \frac{d^4}{dt^4}\frac{1}{24}\|tx\|^4\bigg\vert_{t=0} = \|x\|^4,\eeq so $c_4(r)=c_4(0)=\|\nabla^4\u(0)\|=1$ for all $r\geq0$.

Now we check the assumptions of Corollary~\ref{prop:R1:corr} with $\tau=1$. It is straightforward to show that if $\epsilon\leq1/2$, then~\eqref{tau-ep} is satisfied. Also, Assumptions~\ref{assume:min},~\ref{assume:CL},~\ref{assume:tail} are all clearly satisfied. Finally, the above calculations of the $c_k$ show that~\eqref{ck-0},~\eqref{ck-ev-R} are both satisfied with $\tau=1$ (and~\eqref{ckg-0},~\eqref{ckg-ev-R} are trivially satisfied since $g\equiv1$). Thus we obain $|\Rem_L|\les_L (d^2/n)^L$. To show that $\Rem_L\asymp (d^2/n)^L$, we use that $|A_{2L}n^{-L}|\gtrsim_L (d^2/n)^L$ and $|\Rem_{L+1}|\les_{L}(d^2/n)^{L+1}$, which implies $|\Rem_L|\gtrsim_L (d^2/n)^L$ if $d^2/n$ is small enough. 
\end{proof}
\section{Proofs from Section~\ref{sec:log}}\label{app:log}
\begin{proof}[Proof of Lemma~\ref{lma:adam}]
We have 
$$\nabla^k\u(x)=\frac1n\sum_{i=1}^n\phi^{(k)}(X_i^\T x)X_i^{\otimes k},$$ so that
\beq\label{nablakux}\|\nabla^k\u(x)\|\leq \|\phi^{(k)}\|_\infty\sup_{u\in S^{d-1}}\frac1n\sum_{i=1}^n|X_i^\T u|^k,\eeq where $S^{d-1}$ is the unit sphere in $\R^d$. Now, let $m=\max(d,\sqrt{\log n})$ which ensures $d\leq m\leq n\leq e^{\sqrt m}$. This is needed to apply the result of~\cite{adamczak2010quantitative} cited below. Let $\bar X_i=(X_i, Y_i)$, where $Y_i\iid\mathcal N(0, I_{m-d})$, and the $Y_i$ are independent of the $X_i$. Then $\bar X_i\iid\mathcal N(0, I_m)$ and we have
\beq\label{Xiu}
\sup_{u\in S^{d-1}}\frac1n\sum_{i=1}^n|X_i^\T u|^k \leq \sup_{u\in S^{m-1}}\frac1n\sum_{i=1}^n|\bar X_i^\T u|^k
\eeq with probability 1. Next, by~\cite[Proposition 4.4]{adamczak2010quantitative} applied with $t=s=1$, it holds
\beqs\label{ub-c3}
 \sup_{u\in S^{m-1}}\frac1n\sum_{i=1}^n|\bar X_i^\T u|^k \leq  &\sup_{u\in S^{m-1}}\frac1n\sum_{i=1}^n\E\l[|\bar X_i^\T u|^k\r]\\
 &+C_k\log^{k-1}\l(\frac{2n}{m}\r)\sqrt{\frac mn} + C_k\frac{m^{k/2}}{n} + C_k\frac{m}{n}
\eeqs with probability at least
\beq\label{prob-c3}
1-\exp(-C\sqrt n)-\exp(-C_k\min(n\log^{2k-2}(2n/m), \sqrt{nm}/\log(2n/m)))\geq 1-2\exp(-C_k\sqrt n),
\eeq where $C$ is an absolute constant and $C_k$ depends only on $k$. The inequality in~\eqref{prob-c3} follows by noting that $\log(2n/m)\geq\log 2$ and $\log(2n/m)\leq\log(2e^{\sqrt m})\leq\log(e^{2\sqrt m})=2\sqrt m$, and therefore, $n\log^{2k-2}(2n/m)\geq C_kn$ and $\sqrt{nm}/\log(2n/m)\geq \sqrt n/2$.
We can also further upper bound~\eqref{ub-c3} by using that $m\leq n$ and $\E\l[|\bar X_i^\T u|^k\r]=C_k$ for all $i$ and some $C_k$. Therefore,
\beq\label{further-ub}
 \sup_{u\in S^{m-1}}\frac1n\sum_{i=1}^n|\bar X_i^\T u|^k \leq C_k\l(1+\frac{m^{k/2}}{n}\r) \leq C_k\l(1+\frac{d^{k/2}}{n}\r),
\eeq where $C_k$ may change value. The second inequality uses that $m\leq d+\sqrt{\log n}$. 
Combining~\eqref{nablakux},~\eqref{Xiu}, \eqref{further-ub}, \eqref{prob-c3} gives
$$
\|\nabla^k\u(x)\|\leq C_k\|\nabla^{(k)}\phi\|_\infty\l(1+\frac{d^{k/2}}{n}\r)\quad\text{with prob. at least}\quad 1-2\exp(-C_k\sqrt n),
$$ as desired.
\end{proof}
\paragraph{Derivation of $A_2$.}
Recall that the $A_{2k}$ are given as in Theorem~\ref{lma:leading}, with $f(x)=g(\xmin +H^{-1/2}x)$ and $v(x)=\u(\xmin +H^{-1/2}x)$. For $\partial^{\balph}v(0)$ in the setting of Section~\ref{sec:log}, we have the explicit formula
\beq\label{nablav-log}
\partial^{\balph}v(0) = \frac1n\sum_{i=1}^n\phi^{(|\balph|)}(X_i^\T \xmin )(H^{-1/2}X_i)^{\balph}.\eeq As an example, we now derive the following formula for $A_2$, by applying~\eqref{A2f-explicit} to the present setting:
\beqs\label{A2f-explicit-log}
A_2 =&\frac12\mathrm{Tr}\l(\nabla^2g(\xmin )H^{-1}\r)\\
&-\frac1{2n}\sum_{\ell=1}^n\phi^{(3)}(X_\ell^\T \xmin )\l(\nabla g(\xmin )^\T H^{-1}X_\ell\r)t_{\ell,\ell} \\
&- \frac{g(\xmin )}{8n}\sum_{\ell=1}^n\phi^{(4)}(X_\ell^\T \xmin )t_{\ell,\ell}^2\\
&+ \frac{g(\xmin )}{n^2}\sum_{\ell,m=1}^n\phi^{(3)}(X_\ell^\T \xmin )\phi^{(3)}(X_m^\T \xmin )\l(\frac{1}{12}t_{\ell,m}^3+\frac18t_{\ell,m}t_{\ell,\ell}t_{m,m}\r),
\eeqs
where $t_{\ell,m}=X_\ell^\T H^{-1}X_m$. \\

%
To derive this formula, first note that $\Delta f(0)=\mathrm{Tr}(\nabla^2f(0))=\mathrm{Tr}(H^{-1/2}\nabla^2g(\xmin )H^{-1/2})=\mathrm{Tr}(\nabla^2g(\xmin )H^{-1})$. Next, let $S_\ell=H^{-1/2}X_\ell$, $\ell=1,\dots,n$, and $\phi_{k,\ell}=\phi^{(k)}(X_\ell^\T \xmin )$. Let $S_{\ell,j}$ be the $j$th coordinate of $S_\ell$. We use~\eqref{nablav-log} to derive that
\beq\label{sumijj}
\sum_{j=1}^d\partial_{ijj}^3v(0) = \frac1n\sum_{\ell=1}^n\sum_{j=1}^d\phi_{3,\ell}S_{\ell,i}S_{\ell,j}^2 = \frac1n\sum_{\ell=1}^n\phi_{3,\ell}S_{\ell,i}\|S_\ell\|^2.
\eeq
Therefore, using that $\nabla f(0)=H^{-1/2}\nabla g(\xmin )$ we have
\beqsn
\sum_{i,j=1}^d\partial_if(0)\partial_{ijj}^3v(0) &= \frac1n\sum_{\ell=1}^n\phi_{3,\ell}\nabla f(0)^\T S_{\ell}\|S_\ell\|^2=\frac1n\sum_{\ell=1}^n\phi_{3,\ell}\nabla g(\xmin )^\T H^{-1/2}S_{\ell}\|S_\ell\|^2\\
&=\frac1{n}\sum_{\ell=1}^n\phi_{3,\ell}\l(\nabla g(\xmin )^\T H^{-1}X_\ell\r)\l(X_\ell^\T H^{-1}X_\ell\r).\eeqsn Next, we have
\beqsn
\|\nabla^3v(0)\|^2_{F} &=\sum_{i,j,k=1}^d\l(\frac1n\sum_{\ell=1}^n\phi_{3,\ell}S_{\ell,i}S_{\ell,j}S_{\ell,k}\r)^2=\frac{1}{n^2}\sum_{\ell,m=1}^n\phi_{3,\ell}\phi_{3,m}\sum_{i,j,k=1}^dS_{\ell,i}S_{\ell,j}S_{\ell,k}S_{m,i}S_{m,j}S_{m,k}\\
&=\frac{1}{n^2}\sum_{\ell,m=1}^n\phi_{3,\ell}\phi_{3,m}(S_\ell^\T S_m)^3 = \frac{1}{n^2}\sum_{\ell,m=1}^n\phi_{3,\ell}\phi_{3,m}(X_\ell^\T H^{-1} X_m)^3.
\eeqsn
Next, using~\eqref{sumijj} gives
\beqsn
\sum_{i=1}^d\l(\partial_i\Delta v(0)\r)^2&= \sum_{i=1}^d\l(\frac1n\sum_{\ell=1}^n\phi_{3,\ell}S_{\ell,i}\|S_\ell\|^2\r)^2= \frac{1}{n^2}\sum_{\ell,m=1}^n\phi_{3,\ell}\phi_{3,m}\|S_\ell\|^2\|S_m\|^2\sum_{i=1}^dS_{\ell,i}S_{m,i}\\
&= \frac{1}{n^2}\sum_{\ell,m=1}^n\phi_{3,\ell}\phi_{3,m}\|S_\ell\|^2\|S_m\|^2S_\ell^\T S_m\\
&= \frac{1}{n^2}\sum_{\ell,m=1}^n\phi_{3,\ell}\phi_{3,m}\l(X_\ell^\T H^{-1}X_m\r)\l(X_\ell^\T H^{-1}X_\ell\r)\l(X_m^\T H^{-1}X_m\r).
\eeqsn 
Finally,
\beqsn
\sum_{i,j=1}^d\partial_i^2\partial_j^2v(0) &= \frac1n\sum_{i,j=1}^d\sum_{\ell=1}^n\phi_{4,\ell}S_{\ell,i}^2S_{\ell,j}^2 = \frac1n\sum_{\ell=1}^n\phi_{4,\ell}\|S_\ell\|^4 = \frac1n\sum_{\ell=1}^n\phi_{4,\ell}\l(X_\ell^\T H^{-1}X_\ell\r)^2.\eeqsn
Substituting all of the above estimates into~\eqref{A2f-explicit} gives the desired result.

\section{Lipschitz concentration}\label{app:sec:lip}
\begin{theorem}[Theorem A1 in \cite{schlichting2019poincare}, Theorem 2.1 in~\cite{kolesnikov2016riemannian}]\label{thm:LSI}
Let $\U\subset\R^d$ be a convex set and $\mu$ be a probability measure on $\R^d$ given by $\mu(dx) = Z_\mu^{-1}e^{-H(x)}\ind_\U (x)dx$, where $Z_\mu$ is the normalization constant. Suppose $H\in C^2(\U)$, and $\nabla^2H(x)\succ\alpha I_d$ for all $x\in\U$. Then $\mu$ satisfies $\mathrm{LSI}(\alpha)$, where LSI is the log Sobolev inequality.
\end{theorem}
In particular, therefore, the measure $\gamma_\U $, the restriction of the standard Gaussian $\gamma(dx)\propto e^{-\|x\|^2/2}dx$ to $\U$, satisfies LSI$(1)$ for any convex set $\U$. Note that we consider $\gamma_\U $ to be a measure on $\R^d$. The following is the key theorem which forms the basis for most other bounds derived in this work.
\begin{theorem}[Proposition 5.4.1 in~\cite{bakry2014analysis}]\label{thm:herbst} Let $f:\R^d\to\R$ be M-Lipschitz, and $X\sim\mu$ for a probability measure $\mu$ on $\R^d$ satisfying $\mathrm{LSI}(1)$. Then the random variable $f(X)$ is $M^2$ sub-Gaussian, which by definition means that
\beq\label{exp-subG}
\int e^{\lambda f}d\mu \leq \exp\l(\lambda \int fd\mu +\lambda^2M^2/2\r),\quad\forall\lambda\in\R.
\eeq This in turn implies that
\beq\label{poly-subG}
\int \l|f-\int fd\mu\r|^{q}d\mu \leq (2\sqrt q M)^{q},\quad\forall q\in\mathbb N.
\eeq
\end{theorem} \noindent Note that~\eqref{poly-subG} follows from~\eqref{exp-subG} by a standard Chernoff bound and the formula $\E[|X|^q]=\int_0^\infty qt^{q-1}\mathbb P(|X|\geq t)dt$, applied to the random variable $X=f(Y)-\E[f(Y)]$, $Y\sim\mu$.\\

Below, recall that $\|f\|_q =(\int |f|^qd\gamma)^{1/q}$.
\begin{corollary}\label{corr:gauss}
Suppose $f\in C^1(\bar\U)$ for an open convex set $\U\subseteq\R^d$. Then
\beq\label{corr:exp:gauss}
\|e^f\ind_\U\|_q\leq\exp\l(\frac{\int_\U fd\gamma}{\gamma(\U)} + \frac q2\sup_{x\in\U}\|\nabla f(x)\|^2\r),
\eeq and 
\beq\label{corr:poly:gauss}
\l\| \l(f-\frac{\int_\U fd\gamma}{\gamma(\U)}\r)\ind_\U\r\|_q\leq 2\sqrt q \sup_{x\in\U}\|\nabla f(x)\|.
\eeq
%
%
\end{corollary}
\begin{proof}Let $M=\sup_{x\in\U}\|\nabla f(x)\|$. Since $\U$ is convex, we then have that $f$ is $M$-Lipschitz on $\U$. Let $\tilde f$ be a Lipschitz extension of $f$ to $\R^d$; that is, $f\vert_\U  = \tilde f\vert_\U $, and $\tilde f$ is $M$-Lipschitz on $\R^d$. Such an extension exists by~\cite{lipschitzextension}. Note that if $X\sim\gamma_\U $, then $f(X)$ is equal to $\tilde f(X)$ in distribution. By Theorem~\ref{thm:LSI}, the measure $\gamma_{\U}$ satisfies LSI(1). Therefore by Theorem~\ref{thm:herbst}, we have
$$
\int_\U e^{q f}d\gamma \leq \int e^{q f}d\gamma_\U =\int e^{q \tilde f }d\gamma_\U \leq \exp\l(q\int \tilde fd\gamma_\U + M^2q^2/2\r) = \exp\l(q\int fd\gamma_\U + M^2q^2/2\r).
$$ Taking the $(1/q)$ power of both sides concludes the proof of~\eqref{corr:exp:gauss}. Similarly by~\eqref{poly-subG} we have
$$
\int_\U \l|f-\int fd\gamma_\U \r|^qd\gamma \leq \int \l|f-\int fd\gamma_\U\r|^qd\gamma_\U =\int\l|\tilde f-\int \tilde fd\gamma_\U\r|^qd\gamma_\U \leq (2\sqrt qM)^q.
$$
\end{proof}
For the next proposition, recall the standard Gaussian tail bound
\beq\label{gauss-tail}P(\|Z\|\geq S\sqrt d)\leq \exp\l(-\frac12(S-1)^2d\r).\eeq 
\begin{prop}\label{lma:pk-bds}Define $p:\R^d\to\R$ by $p(x) =\la T, x^{\otimes k}\ra$, for a symmetric order $k$ tensor $T$ on $\R^d$. Also, let $q\geq2$. 
Then 
\begin{align}
\|p\|_q \les_{k,q} \|T\|\times \begin{dcases}{\sqrt d}^{k}, &\text{$k$ even},\\
{\sqrt d}^{k-1}, &\text{$k$ odd}.\end{dcases}
\label{pk-bd}
\end{align}
\end{prop} 
\begin{proof}
If $k$ is even then we use 
\beq\label{pkq}\|p\|_q =\E[|\la T, Z^{\otimes k}\ra|^q]^{1/q}\leq \|T\|(\E[\|Z\|^{kq}])^{\frac1q} \les_{k,q}\|T\|{\sqrt d}^k.\eeq
Next, define $\hat p = p-\int pd\gamma_\V$, where $\gamma_\V$ is the probability density on $\R^d$ proportional to $\gamma\ind_\V$, where
$$\V=\{x\in\R^d\; : \; \|x\|\leq S\sqrt d\},\qquad S=1+\sqrt{q\log(d+1)/d}.$$ Note that using~\eqref{gauss-tail} and $q\geq2$, $d\geq1$, we have $\gamma(\V^c) \leq \exp(-q\log(d+1)/2)\leq\exp(-\log2)=1/2$. Hence $\gamma(\V)\geq1/2$. We will show that $\|\hat p\|_q\les_{k,q}\|T\|{\sqrt d}^{k-1}$ for all $k\geq1$, and then note that $p=\hat p$ when $k$ is odd. 
We have
\begingroup
\addtolength{\jot}{0.5em}
\beqs\label{firstpk}
\|\hat p \|_q &=\l  \|p-\medint\int pd\gamma_{\V}\r \|_q \leq 
 \l \| \l(p-\medint\int pd\gamma_\V\r)\ind_{\V}\r \|_q + \l \| \l(p-\medint\int pd\gamma_\V\r)\ind_{\V^c}\r \|_q.\eeqs\endgroup 
Next, note that
\beqs\label{pkdiff}
 \l \| \l(p-\medint\int pd\gamma_\V\r)\ind_{\V^c}\r \|_q &\leq \|p\|_{2q}\gamma(\V^c)^{1/2}+\gamma(\V)^{-1}\|p\|_{1}\gamma(\V^c)^{1/q}\\
 & \les \|p\|_{2q}\gamma(\V^c)^{1/q}\les_{k,q} \exp(-(S-1)^2d/2q)\|T\|d^{k/2},
 \eeqs using~\eqref{gauss-tail} and~\eqref{pkq}, and the fact that $\gamma(\V)^{-1}\leq2$. For the first term on the righthand side of~\eqref{firstpk}, we use~\eqref{corr:poly:gauss} of Corollary~\ref{corr:gauss} with $f=p$. This gives 
\beq\label{1}
\l \| \l(p-\medint\int pd\gamma_\V\r)\ind_{\V}\r \|_q \les_q \sup_{x\in\V}\|\nabla p(x)\| \les_{q,k}\|T\|(S\sqrt d)^{k-1}.
\eeq 
 Substituting~\eqref{1} and~\eqref{pkdiff} into~\eqref{firstpk} gives
 \beqs\label{123}
 \|\hat p \|_q &\les_{q,k}  \|T\|(S\sqrt d)^{k-1} + \exp(-(S-1)^2d/2q)\|T\|d^{k/2} \\
 &=\|T\|d^{(k-1)/2}\l(S^{k-1} + \sqrt d\exp(-(S-1)^2d/2q)\r).
 \eeqs
Finally, recall that $S=1+\sqrt{q\log (d+1)/d}$, so that $-(S-1)^2d/2q =- \frac12\log (d+1) \leq -\log\sqrt d$, and hence the second term in parentheses is bounded above by 1. Meanwhile,  the first term in parentheses is bounded by some constant depending only on $k$ and $q$. Substituting these bounds into~\eqref{123} concludes the proof.
\end{proof}

\section{Proofs from Section~\ref{sec:proof}}\label{app:sec:proof}
\begin{proof}[Proof of Lemma~\ref{lma:basic-decomp}]By the definition of $f,v$ from $g,\u$, we have
$$
\frac{e^{n\u(\xmin)}\sqrt{\det H}}{(2\pi/n)^{d/2}}\int_{\R^d} g(x)e^{-n\u(x)}dx=(2\pi/n)^{-d/2}\int f(x)e^{-nv(x)}dx.
$$
Next we change variables as $x=n^{-1/2}y$ and split up the integral into two parts:
$$
\int f(x)e^{-nv(x)}dx = n^{-d/2}\int_{\U} f(y/\sqrt n)e^{-nv(y/\sqrt n)}dy + n^{-d/2}\int_{\U^c} f(y/\sqrt n)e^{-nv(y/\sqrt n)}dy.
$$ We now divide through by $(2\pi/n)^{d/2}$, and substitute $$(2\pi)^{-d/2}f(y/\sqrt n)e^{-nv(y/\sqrt n)}dy = F(1/\sqrt n, y)e^{-\rr(1/\sqrt n, y)}\gamma(dy)$$ in the first integral. This gives
\beqsn
(2\pi/n)^{-d/2}&\int f(x)e^{-nv(x)}dx = \int_\U F(1/\sqrt n, y)e^{-\rr(1/\sqrt n, y)}\gamma(dy) + (2\pi)^{-d/2} \int_{\U^c} f(y/\sqrt n)e^{-nv(y/\sqrt n)}dy,
\eeqsn as desired.
\end{proof}
\begin{proof}[Proof of Lemma~\ref{lma:W:Ck}]Fix $x\in\U$ and $t\in(0,1/\sqrt n)$. A Taylor expansion of $v$ gives
\beqsn
v(tx) =\frac{t^2\|x\|^2}{2} +\frac{t^{3}}{2!}\int_0^1\la\nabla^{3}v(qtx), x^{\otimes 3}\ra (1-q)^{2}dq.
\eeqsn Hence for $t\neq 0$, we have
\beqs\label{W-Tay}
\rr(t,x) = \frac{v(tx)}{t^2} - \frac12\|x\|^2=\frac{t}{2!}\int_0^1\la\nabla^{3}v(qtx), x^{\otimes 3}\ra (1-q)^{2}dq
\eeqs Recall that $\rr(0,x)=0$ by definition, which coincides with the righthand side of~\eqref{W-Tay} if we plug in $t=0$. Hence $\rr$ is given by~\eqref{W-Tay} for all $t\in[0,1/\sqrt n)$ and $x\in\U$. Next, we show $\rr$ is differentiable for each $x\in\U$ and $t\in(0,1/\sqrt n)$. Note that by the change of variables $q=s/t$, we can rewrite $\rr$ in the equivalent form
\beq
\rr(t,x) = \frac{t^{-2}}{2}\int_0^t\la\nabla^{3}v(sx), x^{\otimes 3}\ra (t-s)^{2}ds.
\eeq  This function is infinitely differentiable in $t$ for all $t\in(0,1/\sqrt n)$ and $x\in\U$. We compute
\beqs\label{Iab-deriv}
\partial_t\rr(t,x) = &-t^{-3}\int_0^t\la\nabla^{3}v(sx), x^{\otimes 3}\ra (t-s)^{2}ds + t^{-2}\int_0^t\la\nabla^{3}v(sx), x^{\otimes 3}\ra (t-s)ds\\
=&-\int_0^1\la\nabla^{3}v(tqx), x^{\otimes 3}\ra (1-q)^2dq + \int_0^1\la\nabla^{3}v(tqx), x^{\otimes 3}\ra (1-q)dq\\
= &\int_0^1\la\nabla^{3}v(tqx), x^{\otimes 3}\ra q(1-q)dq.
\eeqs  Using that $v\in C^{2L+2}(\bar\U)$ we can take $2L-1$ more derivatives of $t$, passing the derivatives inside the integral. This gives the desired formula~\eqref{partial-t} for $\partial_t^k\rr$. The continuity of these partial derivatives is immediate from the fact that $v\in C^{2L+2}(\bar\U)$.
\end{proof}
Recall from~\eqref{main-ratio} that 
$$\maus{\U^c}=(2\pi)^{-d/2}\int_{\U^c}f(x/\sqrt n)e^{-nv(x/\sqrt n)}dx.$$
\begin{lemma}\label{aux:gamma}
Let $f,v$ satisfy~\eqref{v-av-growth-ii}. Then
$$|\maus{\U^c}| \leq \frac de\e\l(\l[\log\l(4e\r)- \frac\Rad8\r]d\r)\leq \frac de e^{-\Rad d/16},$$ where the second inequality holds if $\Rad \geq 16\log(4e)$.
\end{lemma}
\begin{proof} The result essentially follows from~\cite[Lemma F.1]{katskew}, but we give the proof here to be self-contained. First, note that~\eqref{v-av-growth-ii} implies
$$|f(x/\sqrt n)|\exp(-nv(x/\sqrt n))\leq  \exp(- \sqrt d\|x\|/4)\qquad\forall \|x\|\geq\Rad\sqrt d.$$ Therefore,
\beqsn
|\maus{\U^c}| \leq  I:=(2\pi)^{-d/2}\int_{\|x\|\geq\Rad\sqrt d}e^{- \sqrt d\|x\|/4}dx.\eeqsn 
Switching to polar coordinates and then changing variables, we have
\beqs\label{ISd}
I&=\frac{S_{d-1}}{(2\pi)^{d/2}}\int_{\Rad\sqrt d}^\infty u^{d-1}e^{-\sqrt du/4}du= \frac{S_{d-1}}{(2\pi)^{d/2}(\sqrt d/4)^{d}}\int_{\Rad d/4}^\infty u^{d-1}e^{-u}du,
\eeqs where $S_{d-1}$ is the surface area of the unit sphere. Now, we have 
$$ \frac{S_{d-1}}{(2\pi)^{d/2}} = \frac{2\pi^{d/2}}{\Gamma(d/2)(2\pi)^{d/2}}\leq 2\frac{(2e/d)^{d/2-1}}{2^{d/2}}=\l(\frac ed\r)^{\frac d2-1},$$
using that $\Gamma(d/2)\geq (d/2e)^{d/2-1}$. To bound the integral in~\eqref{ISd}, we use Lemma~\ref{gamma} with $\lambda = \Rad d/4$ and $c=d$. Combining the resulting bound with the above bound on $S_{d-1}/(2\pi)^{d/2}$, we get
\beqs
I &\leq \l(\frac ed\r)^{\frac d2-1}(\sqrt d/4)^{-d}e^{-t\Rad d/4}\l(\frac{d}{1-t}\r)^d= \frac de\l(\frac{\sqrt e}{(1-t)/4}\r)^de^{-t\Rad d/4}
\eeqs Taking $t=1/2$ and noting that $2\sqrt e\leq e$ concludes the proof.
\end{proof}
\begin{lemma}[Lemma F.2 in~\cite{katskew}]\label{gamma}For all $\lambda,c>0$ and $t\in(0,1)$ it holds
$$\int_\lambda^\infty u^{c-1}e^{-u}du\leq  e^{-\lambda t}\l(\frac{c}{1-t}\r)^{c}.$$
\end{lemma}
\begin{proof}
Let $X$ be a random variable with gamma distribution $\Gamma(c, 1)$. Then the desired integral is given by $\Gamma(c)\mathbb P(X\geq\lambda)$. 
Now, the moment generating function of $\Gamma(c,1)$ is $\E[e^{Xt}]=(1-t)^{-c}$, defined for $t<1$. Hence for all $t\in(0,1)$ we have
\beq\label{mgf}
\mathbb P(X\geq\lambda)\leq e^{-\lambda t}(1-t)^{-c}.\eeq  Multiplying both sides by $\Gamma(c)$ and using that $\Gamma(c)\leq c^c$ gives the desired bound.
\end{proof}
\begin{proof}[Proof of Lemma~\ref{lma:alt:bell}]In this proof, we use $[m]$ to denote the set $\{1,2,\dots,m\}$. It is well known~\cite{comtet2012advanced,combinatoricsbook} that for $k\geq1$, we have
\beq
B_k(x_1,\dots, x_k) = \sum_{r=1}^kB_{k,r}(x_1,\dots,x_{k-r+1}),
\eeq where
\beq
B_{k,r}(x_1,\dots, x_{k-r+1})=k! \sum_{\substack{j_1+2j_2+\dots+(k-r+1)j_{k-r+1}=k\\ j_1+\dots+j_{k-r+1}=r \\j_i\geq0\;\forall i\in[k-r+1]}}\;\prod_{i=1}^{k-r+1}\frac{x_i^{j_i}}{(i!)^{j_i}(j_i)!}
\eeq are the partial Bell polynomials. For convenience, we will instead sum over all $j_1,\dots, j_k$ such that $j_1+2j_2+\dots+kj_k=k$ and $j_1+\dots+j_k=r$ and $j_1,\dots,j_k\geq0$ for all $k$. One can show that the only solutions to this equation have $j_1,\dots, j_{k-r+1}$ nonzero, but including $j_i$'s such that $j_i=0$ does not affect the product, since $\frac{x_i^{j_i}}{(i!)^{j_i}(j_i)!}=1$ if $j_i=0$. Therefore another valid formula for $B_{k,r}$ is
\beq\label{x:bell}
B_{k,r}(x_1,\dots, x_{k})=k! \sum_{\substack{j_1+2j_2+\dots+kj_{k}=k\\ j_1+\dots+j_{k}=r \\j_i\geq0\;\forall i\in[k]}}\;\prod_{i=1}^{k}\frac{(x_i/i!)^{j_i}}{(j_i)!},
\eeq
Now, we let $y_i=x_i/i!$ for $i=1,\dots, k$, and separate the numerator from the denominator in~\eqref{x:bell}, to get
\beq\label{y:bell}
B_{k,r}(x_1,\dots, x_{k})=\frac{k!}{r!} \sum_{\substack{j_1+2j_2+\dots+kj_{k}=k\\ j_1+\dots+j_{k}=r \\j_i\geq0\;\forall i\in[k]}}{r\choose{j_1,\dots,j_k}}\;\prod_{i=1}^{k}y_i^{j_i}
\eeq
Consider the product $\prod_{i=1}^{k}y_i^{j_i}$ in~\eqref{y:bell} for a fixed choice of $j_1,\dots,j_{k}$ satisfying the constraints. Note that there are $j_1$ copies of $y_1$, $j_2$ copies of $y_2$ and so on, up to $j_{k}$ copies of $y_k$. There are $r$ total terms in the product, including copies. For example if $k= 11$, $r=6$,  $j_1=3,j_2=1,j_3=2$ and the other $j_i$'s are zero, then the terms arising in the product are $$y_1,y_1,y_1,y_2,y_3,y_3.$$ Note that there are $r!/j_1!\dots j_k!$ permutations of this list which correspond to the same set of $y_i$'s and therefore to the same product. To each permutation, assign $m_i$ to be the index of the $y$ in position $i$. In the example given above, in that order, we would have  $m_1=1,m_2=1,m_3=1,m_4=2,m_5=3,m_6=3$. For the permutation $$y_1, y_3,y_2,y_3,y_1,y_1,$$ we would take $m_1=1,m_2=3,m_3=2,m_4=3,m_5=1,m_6=1$. For each assignment, we have $m_1+\dots+m_r=k$, and
$$\prod_{i=1}^ky_i^{j_i} =\prod_{\ell=1}^ry_{m_\ell}.$$ 
Let $\mathcal M(j_1,\dots,j_k)$ be the set of distinct assignments $m_1,\dots, m_r$; we know that $|\mathcal M(j_1,\dots,j_k)| = r!/j_1!\dots j_k!$. We conclude that
$${r\choose{j_1,\dots,j_k}}\;\prod_{i=1}^{k}y_i^{j_i} = \sum_{(m_1,\dots,m_r)\in \mathcal M(j_1,\dots,j_k)}\prod_{\ell=1}^ry_{m_\ell}.$$ Finally, it is clear that for distinct $j_1,\dots, j_k$, the sets $\mathcal M(j_1,\dots,j_k)$ are disjoint, and
$$\bigcup_{\substack{j_1+2j_2+\dots+kj_{k}=k\\ j_1+\dots+j_{k}=r \\j_i\geq0\;\forall i\in[k]}}\mathcal M(j_1,\dots,j_k) = \l\{(m_1,\dots, m_r)\; :\; m_1+\dots+m_r=k, 1\leq m_\ell \leq k\;\forall\ell\in[r]\r\}.$$
In fact, if $m_1+\dots +m_r=k$ and $m_\ell\geq1$ for all $\ell$, then it automatically follows that $m_\ell\leq k$ for all $\ell$, so we can omit this condition. We conclude that
\beqsn
B_{k,r}(x_1,\dots, x_{k})=\frac{k!}{r!} \sum_{\substack{j_1+2j_2+\dots+kj_{k}=k\\ j_1+\dots+j_{k}=r \\j_i\geq0\;\forall i\in[k]}}{r\choose{j_1,\dots,j_k}}\;\prod_{i=1}^{k}y_i^{j_i}=\frac{k!}{r!} \sum_{\substack{m_1+\dots+m_r=k\\m_\ell\geq1\;\forall \ell\in[r]}}\prod_{\ell=1}^ry_{m_\ell}.\eeqsn
Summing over all $1\leq r\leq k$ and recalling that $y_m=x_m/m!$ concludes the proof.
\end{proof}
\begin{proof}[Proof of Theorem~\ref{lma:leading}]For ease of reference, we recall~\eqref{Akr0}:
\beqs\label{Akr}
A_k = &\frac{1}{k!}\E\l[\la\nabla^{k}f(0), Z^{\otimes k}\ra\r] \\
&+\sum_{\ell=1}^k\sum_{r=1}^\ell\frac{(-1)^r}{r!}\;\sum_{\substack{m_1+m_2+\dots+m_r=\ell\\ m_1,\dots,m_r\geq1}}\E\l[\frac{\la\nabla^{k-\ell}f(0), Z^{\otimes k-\ell}\ra}{(k-\ell)!}\prod_{j=1}^r\frac{\la\nabla^{m_j+2}v(0), Z^{\otimes m_j+2}\ra}{(m_j+2)!}\r]
\eeqs 
We also recall the notation on multi-indices $\balph$ introduced before the theorem statement. Consider the inner product $\la T, S\ra$ for two symmetric tensors $T,S$ of order $m$; that is, $T=(T_{i_1i_2\dots i_m})_{i_1,\dots,i_m=1}^d$ and similarly for $S$. Consider the set of multi-indices $\balph=(\alpha^1,\dots,\alpha^d)$ such that $\alpha^j\geq0$, $j=1,\dots, d$ and $|\balph|:=\sum_{j=1}^d\alpha^j = m$. We define the map $\bm i=(i_1,\dots, i_m)\mapsto \balph=(\alpha^1,\dots,\alpha^d)$ according to $\alpha^j(\bm i) = \sum_{k=1}^m\delta_{i_k, j}$. In other words, $\alpha^j(\bm i)$, $j=1,\dots, d$ is the number of times $j$ appears among the indices $i_1,\dots, i_m$. Note that to each $\balph$ correspond $m!/\balph!$ different $\bm i$'s such that $\balph(\bm i)=\balph$. Finally, since $T$ is symmetric we have $T_{\bm i}=T_{\bm i'}$ for any $\bm i,\bm i'$ such that $\balph(\bm i)=\balph(\bm i')$. We therefore define $T^\balph:=T_{\bm i}$ for any $\bm i$ such that $\balph(\bm i)=\balph$, and similarly for $S$. These considerations imply that the inner product $\la T, S\ra$ can be written as
$$
\la T, S\ra = \sum_{i_1,\dots, i_m=1}^dT_{i_1\dots i_m}S_{i_1\dots i_m} =  \sum_{|\balph|=m}\frac{m!}{\balph!}T^\balph S^\balph.
$$ Furthermore, if $T=\nabla^mv(0)$ and $S=Z^{\otimes m}$, then $T^\balph = \partial^\balph v(0)$ and $S^\balph = Z^\balph$. We therefore have
$$
\frac{\la\nabla^{m_\ell+2}v(0), Z^{\otimes m_\ell+2}\ra}{(m_\ell+2)!} = \frac{1}{(m_\ell+2)!}\sum_{|\balph|=m_\ell+2}\frac{(m_\ell+2)!}{\balph!}\partial^\balph v(0)Z^\balph = \sum_{|\balph|=m_\ell+2}\frac{1}{\balph!}\partial^\balph v(0)Z^\balph,
$$ and similarly
$$
\frac{\la\nabla^{k-\ell}f(0), Z^{\otimes k-\ell}\ra}{(k-\ell)!} = \sum_{|\bbet|=k-\ell}\frac{1}{\bbet!}\partial^{\bbet}f(0)Z^{\bbet},
$$ so that
\beqs\label{prod-ell}
&\frac{\la\nabla^{k-\ell}f(0), Z^{\otimes k-\ell}\ra}{(k-\ell)!}\prod_{\ell=1}^r\frac{\la\nabla^{m_\ell+2}v(0), Z^{\otimes m_\ell+2}\ra}{(m_\ell+2)!} \\
&= \sum_{|\bbet|=k-\ell, |\balph_1|=m_1+2,\dots,|\balph_r|=m_r+2}\,\frac{1}{\bbet!\balph_1!\dots\balph_r!}\;\partial^{\bbet}f(0)\partial^{\balph_1} v(0)\dots\partial^{\balph_r}v(0)Z^{\bbet+\balph_1+\dots+\balph_r}.
\eeqs 
Finally, note that
\beqsn
\E\l[Z^{\bbet+\balph_1+\dots+\balph_r}\r] &= \prod_{i=1}^d\E\l[Z_i^{\beta^i+\alpha_1^i+\dots+\alpha_r^i}\r] \\
&= \begin{cases}(\bbet+\balph_1+\dots+\balph_r-\mathbf{1})!!,\quad &\beta^i+\alpha_1^i+\dots+\alpha_r^i\;\text{even for all $i$},\\
0,\quad&\text{otherwise}.\end{cases}\eeqsn
 Thus recalling the definition of $\even(\balph)$ from the statement of the theorem, we conclude that \beq\label{EZalph}\E\l[Z^{\bbet+\balph_1+\dots+\balph_r}\r] = (\bbet+\balph_1+\dots+\balph_r-\mathbf{1})!!\,\even(\bbet+\balph_1+\dots + \balph_r).\eeq Using~\eqref{prod-ell} and~\eqref{EZalph}, we now have
\beqsn\label{Ak-summand}
\E\bigg[\frac{\la\nabla^{k-\ell}f(0), Z^{\otimes k-\ell}\ra}{(k-\ell)!}\prod_{\ell=1}^r&\frac{\la\nabla^{m_\ell+2}v(0), Z^{\otimes m_\ell+2}\ra}{(m_\ell+2)!}\bigg]\\
= \sum_{\substack{ |\balph_1|=m_1+2,\dots,|\balph_r|=m_r+2\\ |\bbet|=k-\ell}}&\bigg\{\;\frac{(\bbet+\balph_1+\dots+\balph_r-\mathbf{1})!!}{\bbet!\balph_1!\dots\balph_r!}\partial^\bbet f(0)\partial^{\balph_1} v(0)\dots\partial^{\balph_r}v(0)\\
&\times\even(\bbet+\balph_1+\dots + \balph_r)\bigg\}
\eeqsn By a similar argument,
$$
\frac{1}{k!}\E\l[\la\nabla^kf(0), Z^{\otimes k}\ra\r]=\sum_{|\bbet|=k}\frac{(\bbet-\mathbf{1})!!}{\bbet!}\partial^\bbet f(0)\even(\bbet).
$$
Substituting these expressions into~\eqref{Akr} concludes the proof.
\end{proof}

\section{Auxiliary results and postponed proofs}\label{app:aux:bell}
\begin{lemma}\label{lma:cvx}
Let $\u$ be convex with a critical point at zero, and $H=\nabla^2\u(0)\succ0$. Let \\$c_3( r)=\sup_{\|x\|_H\leq r\sqrt{d/n}}\|\nabla^3\u(x)\|_H$ for some $r>0$. Then 
$$\u(x)-\u(0)\geq \l(\frac{ r}{2} -  \frac{ r^2c_3( r)}{6}\sqrt{d/n}\r)\sqrt{d/n}\|x\|_H,\qquad\forall\;\|x\|_H\geq r\sqrt{d/n}.$$
\end{lemma}
\begin{proof}
Fix $x$ such that $\|x\|_H\geq r\sqrt{d/n}$, and let $y=tx$ for $t\in(0,1]$ be chosen to ensure $\|y\|_H= r\sqrt{d/n}$. By convexity of $\u$ we have
\beq\label{check:log}
\u(x)-\u(0) \geq \frac{\u(y)-\u(0)}{t} = \frac{\u(y)-\u(0)}{\|y\|_H}\|x\|_H \geq \bigg\{\inf_{\|y\|_H= r\sqrt{d/n}}\frac{\u(y)-\u(0)}{\|y\|_H}\bigg\}\|x\|_H.
\eeq
Next, fix any $y$ such that $\|y\|_H= r\sqrt{d/n}$. By a Taylor expansion of $\u$ about zero there exists some $t=t(y)$ such that 
\beqsn
\frac{\u(y)-\u(0)}{\|y\|_H} &= \frac{\frac12\|y\|_H^2 + \frac16\la\nabla^3\u(ty), y^{\otimes 3}\ra}{\|y\|_H} \geq \frac12\|y\|_H - \frac16\|\nabla^3\u(ty)\|_H\|y\|_H^2\\
&\geq \frac12\|y\|_H - \frac{c_3(1)}{6}\|y\|_H^2 = \frac{ r}{2}\sqrt{d/n} - \frac{ r^2c_3( r)}{6}d/n.
\eeqsn Substituting this bound into~\eqref{check:log} gives the desired result. \end{proof}
\begin{lemma}\label{lma:Bellprops}
Let $t_1,\dots,t_\ell,\beta\geq0$. Then
\beqs
\beta^{\ell}B_\ell(t_1,\dots,t_\ell)&=B_\ell(\beta^1 t_1,\beta^2t_2,\dots,\beta^\ell t_\ell).
\eeqs
\end{lemma}
\begin{proof}
Using the definition~\eqref{def:bell} of the Bell polynomials, we have
\beqsn
\beta^{\ell}B_\ell(t_1,\dots,t_\ell)  &= \beta^\ell\ell!\sum_{\substack{j_1+2j_2+\dots+\ell j_\ell=\ell\\ j_1,\dots,j_\ell\geq0}}  \prod_{i=1}^\ell\frac{t_i^{j_i}}{j_i!(i!)^{j_i}}\\
&=\ell!\sum_{\substack{j_1+2j_2+\dots+\ell j_\ell=\ell\\ j_1,\dots,j_\ell\geq0}}  \prod_{i=1}^\ell\frac{(\beta^it_i)^{j_i}}{j_i!(i!)^{j_i}} = B_\ell(\beta^1 t_1,\beta^2t_2,\dots,\beta^\ell t_\ell),\eeqsn as desired. 
\end{proof}
\begin{lemma}\label{lma:Bell:Lp}Let $f_1,\dots,f_\ell:\R^d\to\R$ and $\Omega\subset\R^d$, with $f_i\ind_\Omega\in L^{p\ell^2}(\gamma)$, $i=1,\dots,\ell$. Then
\beqs
\|B_\ell(f_1,\dots,f_\ell)\,\ind_\Omega\|_p &\leq  B_\ell(\|f_1\ind_\Omega\|_{p\ell^2},\dots,\|f_\ell\ind_\Omega\|_{p\ell^2}).
\eeqs
\end{lemma}
\begin{proof}
Using the definition~\eqref{def:bell} of the Bell polynomials, we have
\beqs
\|B_\ell(f_1,\dots,f_\ell) \ind_\Omega\|_{p} \leq \ell!\sum_{\substack{j_1+2j_2+\dots+\ell j_\ell=\ell\\ j_1,\dots,j_\ell\geq0}}\l \| \prod_{i=1}^\ell f_i^{j_i}\ind_\Omega  \r \|_{p}\prod_{i=1}^\ell\frac{1}{j_i!(i!)^{j_i}}.\eeqs  Next, using the generalized H\"{o}lder's inequality, we have
$$ \l \| \prod_{i=1}^\ell f_i^{j_i} \ind_\Omega \r \|_{p} = \l \| \prod_{i=1}^\ell f_i^{pj_i}\ind_\Omega \r\|_1^{1/p} \leq \prod_{i=1}^\ell \|f_i^{pj_i}\ind_\Omega\|_{\ell}^{1/p}= \prod_{i=1}^\ell\l \| f_i \ind_\Omega\r \|_{p\ell j_i}^{j_i} \leq  \prod_{i=1}^\ell\l \| f_i \ind_\Omega\r \|_{p\ell^2}^{j_i},$$ noting that $j_i\leq\ell$ for all $j_i$ appearing in the sum. We conclude that
$$
\|B_\ell(f_1,\dots,f_\ell)\ind_\Omega \|_{p} \leq \ell!\sum_{\substack{j_1+2j_2+\dots+\ell j_\ell=\ell \\ j_1,\dots,j_\ell\geq0}}\; \prod_{i=1}^\ell \frac{\|f_i \ind_\Omega\|_{p\ell^2}^{j_i}}{j_i!(i!)^{j_i}} = B_\ell(\|f_1\ind_\Omega\|_{p\ell^2},\dots,\|f_\ell\ind_\Omega\|_{p\ell^2}),
$$ as desired.
\end{proof}

\bibliographystyle{plain} 
\bibliography{bibliogr_Lap_expansion}   
\end{document}